\documentclass[a4paper,11pt]{amsart} 
\usepackage{amsmath}

\usepackage{nccmath}
\usepackage{amssymb,amsthm,mathtools}

\usepackage{mathrsfs}

\usepackage{mathabx}
\usepackage{filecontents}
\usepackage{hyperref}
\usepackage{enumitem}
\usepackage[all]{xy}
\xyoption{line}
\usepackage{color}
\usepackage[normalem]{ulem}

\renewcommand{\le}{\varleq}
\renewcommand{\ge}{\vargeq}

\newcommand{\myforall}{\text{ for all }}

\newcommand{\myand}{\text{ and }}
\newcommand{\myif}{\text{ if }}

\newcommand{\mythen}{\text{ then }}

\newcommand{\seb}{\{\,}
\newcommand{\sen}{\,\}}

\newcommand{\getsby}[1]{\xleftarrow{#1}}

\newcommand{\tesgh}{edge-surjective graph homomorphism}

\newcommand{\pdirectional}{\raise0.05em\hbox{$+$}directional}
\newcommand{\pdirectionality}{\raise0.05em\hbox{$+$}directionality}
\newcommand{\pdirectionalitys}{\raise0.05em\hbox{$+$}directionality }
\newcommand{\pdirectionals}{\raise0.05em\hbox{$+$}directional }
\newcommand{\mdirectional}{\raise0.05em\hbox{$-$}directional}
\newcommand{\mdirectionality}{\raise0.05em\hbox{$-$}directionality}
\newcommand{\mdirectionalitys}{\raise0.05em\hbox{$-$}directionality }
\newcommand{\mdirectionals}{\raise0.05em\hbox{$-$}directional }
\newcommand{\bidirectional}{bidirectional}
\newcommand{\bidirectionals}{bidirectional }

\newcommand{\bidirectionalitys}{bidirectionality }

\newcommand{\Z}{\mathbb{Z}}
\newcommand{\Nonne}{\mathbb{N}}
\newcommand{\Posint}{\mathbb{N}^+}
\newcommand{\Real}{\mathbb{R}}
\newcommand{\bi}{\in \Z}

\newcommand{\beposint}{\in \Posint}
\newcommand{\bpi}{\ge 1}

\newcommand{\benonne}{\in \Nonne}
\newcommand{\bni}{\ge 0}

\newcommand{\diam}{{\rm diam}}

\newcommand{\pstrzinf}[1]{(#1_0,#1_1,#1_2,\dotsc)}

\newcommand{\num}[1]{\abs {#1} }

\newcommand{\ep}{\varepsilon}

\newcommand{\Ccal}{\mathcal{C}}
\newcommand{\Dcal}{\mathcal{D}}

\newcommand{\Gcal}{\mathcal{G}}

\newcommand{\Pcal}{\mathcal{P}}

\newcommand{\Scal}{\mathcal{S}}

\newcommand{\Ucal}{\mathcal{U}}

\newcommand{\sC}{\mathscr{C}}

\newcommand{\sM}{\mathscr{M}}

\newcommand{\sO}{\mathscr{O}}
\newcommand{\sP}{\mathscr{P}}

\newcommand{\kuu}{\emptyset}
\newcommand{\nekuu}{\neq \kuu}

\newcommand{\fai}{\varphi}

\newcommand{\bN}{\boldsymbol{N}}

\newcommand{\barp}{\bar{p}}

\newcommand{\barr}{\bar{r}}
\newcommand{\bars}{\bar{s}}
\newcommand{\bart}{\bar{t}}

\newcommand{\barw}{\bar{w}}
\newcommand{\barx}{\bar{x}}
\newcommand{\bary}{\bar{y}}
\newcommand{\barz}{\bar{z}}

\newcommand{\barX}{\bar{X}}

\newcommand{\ddx}{\ddot{x}}

\newcommand{\ddX}{\ddot{X}}

\newcommand{\centb}{\begin{center}}
\newcommand{\centn}{\end{center}}

\newcommand{\enumb}{\begin{enumerate}}
\newcommand{\enumn}{\end{enumerate}}

\newcommand{\itemb}{\begin{itemize}}
\newcommand{\itemn}{\end{itemize}}

\numberwithin{equation}{section}

\usepackage{cleveref}
\setlist[enumerate,1]{label=(\alph*),ref=(\alph*)}
\setlist[enumerate,2]{label=(\arabic*),ref=(\alph{enumi}-\arabic{enumii})}
\setlist[enumerate,3]{label=(\Alph*),ref=(\roman{enumi}-\alph{enumii}-\Alph*)}
\setlist[enumerate,4]{label=(\arabic*),ref=(\roman{enumi}-\alph{enumii}-\Alph{enumiii}-\arabic*)}

\newlist{alphaenum}{enumerate}{1}
\setlist[alphaenum,1]{label=($\alpha$-\arabic*),ref=($\alpha$-\arabic*)}
\newlist{betaenum}{enumerate}{1}
\setlist[betaenum,1]{label=($\beta$-\arabic*),ref=($\beta$-\arabic*)}
\newlist{gammaenum}{enumerate}{1}
\setlist[gammaenum,1]{label=($\gamma$-\arabic*),ref=($\gamma$-\arabic*)}

\newtheorem{thm}{Theorem}[section]
\newtheorem{lem}[thm]{Lemma}
\newtheorem{prop}[thm]{Proposition}

\theoremstyle{definition}
\newtheorem{defn}[thm]{Definition}

\theoremstyle{remark}

\newtheorem{nota}[thm]{Notation}
\newtheorem{rem}[thm]{Remark}

\crefname{sec}{\S}{\S\S}
\crefname{mainthm}{Theorem}{Theorems}
\crefname{thm}{Theorem}{Theorems}
\crefname{lem}{Lemma}{Lemmas}
\crefname{prop}{Proposition}{Propositions}
\crefname{cor}{Corollary}{Corollaries}
\crefname{defn}{Definition}{Definitions}
\crefname{conj}{Conjecture}{Conjectures}
\crefname{example}{Example}{Examples}
\crefname{nota}{Notation}{Notations}
\crefname{rem}{Remark}{Remarks}
\crefname{note}{Note}{Notes}
\crefname{case}{Case}{Cases}
\crefname{figure}{Figure}{Figures}
\crefname{section}{\S}{\S\S}
\crefname{enumi}{}{}
\crefname{enumii}{}{}
\crefname{equation}{}{}

\newcommand{\Deltainf}{\Delta_{\infty}}

\newcommand{\ov}{\overline}
\newcommand{\ovCcalinf}{\ov{\Ccal}_{\infty}}

\newcommand{\tilc}{\tilde{c}}

\newcommand{\tile}{\tilde{e}}

\newcommand{\abs}[1]{\lvert#1\rvert}

\newcommand{\gap}{{\rm gap}}

\newcommand{\Vp}{V \setminus V_0}

\newcommand{\covrepa}[2]{#1_0 \getsby{#2_1} #1_1 \getsby{#2_2} #1_2 \getsby{#2_3} \dotsb}

\newcommand{\Scalna}{\Scal_{\rm na}}
\newcommand{\Scalue}{\Scal_{\rm ue}}

\usepackage[left=30mm,right=30mm,top=25mm,bottom=25mm]{geometry}

\begin{document}

\title[proximal Cantor systems with rank 2]
{proximal Cantor systems with topological rank 2\\
 are residually scrambled}

\author{Takashi Shimomura}

\address{Nagoya University of Economics,
 Uchikubo 61-1, Inuyama 484-8504, Japan}
\curraddr{}
\email{tkshimo@nagoya-ku.ac.jp}
\thanks{}



\date{\today}

\dedicatory{}

\commby{}

\begin{abstract}
Downarowicz and Maass (2008) proposed topological ranks
 for all homeomorphic Cantor minimal dynamical systems
 using properly ordered Bratteli diagrams.
In this study, we adopt this definition to the case of all essentially
 minimal zero-dimensional systems.
We consider the cases in which topological ranks are 2 and unique
 minimal sets are fixed points.
Akin and Kolyada (2003), in their study of Li--Yorke sensitivity,
 showed that if the unique minimal set of an essentially minimal system is
 a fixed point,
 then the system must be proximal.
However, a finite topological rank implies expansiveness;
 furthermore, in the case of proximal Cantor systems with topological rank 2,
 the expansiveness is always from the lowest degree.

Rank 2 zero-dimensional systems might be thought as a part of the rank 1 transformations that are considered in the vast field of ergodic theory.
However, these systems are also interesting from the perspective of
 topological chaos theory;
 e.g., in this study, we show that
 all proximal Cantor systems with topological rank 2
 are residually scrambled.
In addition, we investigate the finite invariant measures on these systems.
Evidently, such systems have at most two ergodic measures.
We present a necessary and sufficient condition
 for the unique ergodicity of these systems.
In addition, we show that the number of ergodic measures of systems that are topologically mixing can be 1 and 2.
Moreover, we present examples that are topologically weakly mixing, not topologically mixing, and uniquely ergodic.
Finally, we show that the number of ergodic measures of systems
 that are not weakly mixing can be 1 and 2.
\end{abstract}

\maketitle

\section{Introduction}\label{sec:introduction}
In this paper,
 a pair $(X,f)$ of a compact metric space $X$ and a homeomorphism $f : X \to X$
 is called a \textit{(homeomorphic) topological dynamical system}.
If $X$ is totally disconnected,
 then $(X,f)$ is called a \textit{(homeomorphic) zero-dimensional system};
 further, if $X$ is homeomorphic to the Cantor set,
 then $(X,f)$ is called a \textit{(homeomorphic) Cantor system}.
If $f$ is a continuous surjective map,
 then we explicitly state that
 $(X,f)$ is a continuous surjective topological dynamical system,
 continuous surjective zero-dimensional system, and so on.
The minimal zero-dimensional systems are thoroughly investigated
 in the connection with the theory of $C^*$-algebras.
In a preparatory study, Herman, Putnam, and Skau
 \cite{HERMAN_1992OrdBratteliDiagDimGroupTopDyn}
showed
 \cite[Theorem 4.7]{HERMAN_1992OrdBratteliDiagDimGroupTopDyn}
 states in which
a bijective correspondence exists between
 equivalence classes of essentially simple ordered Bratteli diagrams
 and pointed topological conjugacy classes
 of essentially minimal zero-dimensional systems.
By the word essentially minimal zero-dimensional systems,
 they meant zero-dimensional systems with unique minimal sets. 
In their Bratteli--Vershik model,
 the maximal and minimal paths have to be selected from the minimal sets.
From this, the properly ordered Bratteli diagrams have been used to represent
 minimal Cantor systems, and the maximal and minimal paths are
 selected arbitrarily from the total spaces.
Innumerable works have been done on this model.
In particular,
 Downarowicz and Maass
 in \cite{DOWNAROWICZ_2008FiniteRankBratteliVershikDiagAreExpansive}
 defined the topological ranks for all homeomorphic Cantor minimal systems.
Further, they demonstrated that all homeomorphic Cantor minimal systems
 with topological rank $K > 1$ are expansive.
Thus, defining the topological ranks based on the original essentially simple
 Bratteli diagrams is natural.
We do this for all homeomorphic essentially minimal zero-dimensional systems
 with respect to the essentially simple ordered Bratteli diagrams.
Akin and Kolyada
 \cite[Proposition 2.2]{Akin_2003LiYorkeSens}
 characterized all proximal systems as systems having unique minimal
 sets that are fixed points.
Accordingly, homeomorphic proximal zero-dimensional systems
 are characterized as
 essentially simple Bratteli--Vershik systems
 that have unique fixed points; further, they have topological ranks.
Zero-dimensional proximal systems with finite topological ranks are 
 symbolic (see \cite{Shimomura_2017FiniteRankBVwithPeriodicityExpansive}).
In particular, zero-dimensional proximal systems with topological rank 2
 are symbolic.
We show that the degree of expansiveness in the Bratteli--Vershik representation is always from the lowest degree (see \cref{prop:expansive-from-1}).

In \cite[Proposition 55]{Blanchard_2008TopSizeOfScrambSet},
 Blanchard, Huang, and Snoha
 presented an example of substitution dynamics that is 
 residually scrambled
 (see
 \cref{defn:proximal,defn:Li-Yorke,defn:scrambled,defn:residually-scrambled}).
In finding this proposition, we have conducted a brief survey on homeomorphic proximal Cantor systems with topological rank 2.
In this proposition, they also showed that the system is topologically mixing
 and has some ergodic properties.
We could clarify that the substitution system constructed in
 \cite[Proposition 55]{Blanchard_2008TopSizeOfScrambSet}
 is a proximal
 (see \cref{defn:proximal-system})
 Cantor system with topological rank 2
 (see \cref{defn:rank} and \cref{prop:proposition55-is-rank-2}).
We extend the study to proximal Cantor systems with topological
 rank 2 and show that all proximal Cantor systems with
 topological rank 2 are residually scrambled
 (see \cref{thm:residually-scrambled}).
We also present a necessary and sufficient condition for the 
unique ergodicity of such systems (see \cref{thm:uniquely-ergodic}).

In studying such systems,
 we adopt the graph covering approach in some special forms
 that can easily be
 translated into the Bratteli--Vershik approach.
On this point, Bernardes and Darji in \cite{BernardesDarji_2012GraphTheoreticStructureOfMapsOfTheCantorSpace} had already used the technique using finite directed graphs and its covering maps for the study of continuous surjective or homeomorphic Cantor systems.
The technique has also been used in some works.
For example, Fernandez, Good, and Puljiz \cite{FernandezGoodPuljiz_2017Almost_minimal_systems_and_periodicity_in_hyperspaces} constructed an almost totally minimal homeomorphism of the Cantor set.
Boro{\'{n}}ski, Kupka, and Oprocha \cite{BoronskiKupkaOprocha_2017AMixingCompletelyScrambledSystemExists} showed that there exists a completely scrambled topologically mixing system. 
By the similar approach, we introduce a special kind of graph coverings,
 i.e. graph coverings of Kakutani--Rohlin type that are abbreviated as
 KR-coverings.
With this, we can define the topological rank that matches with
 the one that had been defined onto minimal Cantor systems by Downarowicz
 and Maass \cite{DOWNAROWICZ_2008FiniteRankBratteliVershikDiagAreExpansive}
 (see \cref{thm:coincidence-of-ranks-essentially-simple-BV-and-KR-cov}).

One of the properties of proximal Cantor systems with topological rank 2 identified in this paper is the existence of residual scrambled sets.
Hence, let us recall some basic definitions of topological dynamical systems.
From the viewpoint of Li--Yorke chaotic systems, the size of scrambled sets
 has been discussed in various aspects.
In \cite{Gedeon_1987ThereNoChaotMapWithResidScramSets}, Gedeon
 found that there is no residually scrambled system in the maps of
 the closed interval.
On the other hand, various completely scrambled systems
 (see \cref{defn:completely-scrambled}) have been found
 in zero-dimensional spaces and in spaces of dimension $1 \le n \le \infty$
 (\cite{HUANG_2001HomeoWithWholeCompBeingScrambledSets,Fory_2016InvScramSetsUnifRigidWeakMix,BoronskiKupkaOprocha_2017AMixingCompletelyScrambledSystemExists,Shimomura_2015completelyscrambled}).
Nevertheless, being completely scrambled is a very tight restriction on the
 system.
In \cite{Blanchard_2008TopSizeOfScrambSet}, Blanchard, Huang, and Snoha
 presented various examples of residually scrambled systems in addition to 
 the example mentioned at the beginning of this paper.
In particular, they constructed
 a residually scrambled system with positive topological entropy
 from any proximal residually scrambled system having an invariant measure
 that has positive measure on every non-empty open set
 (see \cite[Proposition 56]{Blanchard_2008TopSizeOfScrambSet}).
Therefore,
 in this paper, we try to extend the above-mentioned example
 \cite[Proposition 55]{Blanchard_2008TopSizeOfScrambSet},
 which is a non-primitive substitution dynamical system of two symbols
 to a larger class of proximal Cantor systems 
 with topological rank 2.
In fact, we show that every proximal Cantor system with topological rank 2
 is topologically conjugate to a residually scrambled symbolic system of
 two symbols (\cref{thm:residually-scrambled}).
There exist at most two ergodic measures; if two exist, 
 exactly one of them has positive measure on every non-empty open set.
Further, using the graph covering method,
 we show that the number of ergodic measures of systems
 that are topologically mixing can be both 1 and 2.
 Moreover, we present examples that are topologically weakly mixing, not topologically mixing, 
 and uniquely ergodic.
Finally, we show that the number of ergodic measures of systems
 that are not weakly mixing can be both 1 and 2.

The remainder of this paper is organized as follows. First, some preliminaries are introduced in \cref{sec:preliminaries}.
We investigate only finite invariant measures.
 Next, the Bratteli--Vershik representation of zero-dimensional essentially minimal 
homeomorphisms is presented in \cref{sec:Bratteli-Vershik-representations}, and a link is established between the Bratteli--Vershik representation
approach and the graph covering approach in the case of zero-dimensional essentially minimal systems. In \cref{sec:proximal-systems}, we state
\cref{thm:proximal}, which characterizes zero-dimensional proximal systems
by means of Bratteli--Vershik systems and graph coverings. In \cref{sec:array-systems}, we present the array system approach developed by Downarowicz and Maass \cite{DOWNAROWICZ_2008FiniteRankBratteliVershikDiagAreExpansive}.
Finally, in \cref{sec:proximal-systems-of-rank-2}, we present the results of our survey on proximal Cantor systems with topological rank 2
as well as some examples.
\section{Preliminaries}\label{sec:preliminaries}
Let  $\Z$ be the set of all integers
 and $\Nonne$ be the set of all non-negative integers.
For integers $a < b$,
 the intervals are denoted by $[a,b] := \seb a, a+1, \dotsc,b \sen$.
A continuous surjective topological dynamical system $(X,f)$
 is \textit{topologically transitive}
 if there exists an $x \in X$
 such that $\seb f^n(x) \mid n \bni \sen$ is dense in $X$.
If $(X,f)$ is topologically transitive and $\num{X}$ is not finite,
 then $X$ does not have isolated points.
When $f$ is a homeomorphism,
 a point $x \in X$ is a \textit{forward transitive point}
 if $\seb f^n(x) \mid n \bni \sen$ is dense in  $X$.
The notion of backward transitive points is also defined in the natural sense.
A topological dynamical system $(X,f)$ is \textit{topologically mixing}
 if for any opene $U,V$, there exists $N > 0$ such that $f^n(U) \cap V \nekuu$
 for all $n \ge N$.
A topological dynamical system $(X,f)$ is \textit{weakly mixing}
 if $(X\times X,f\times f)$ is transitive.
A zero-dimensional system $(X,f)$
 is called a Cantor system if $X$ is homeomorphic
 to the Cantor set, i.e., $X$ does not have isolated points.
We say that
 $(X,f)$ is \textit{expansive} with an \textit{expansive constant} $\ep > 0$
 if for every pair $(x,y) \in X^2 \setminus \Delta_X$,
 there exists an $n \bi$ such that $d(f^n(x),f^n(y)) \ge \ep$.
Let $A$ be a finite set.
An element $x \in A^{\Z}$ is written as $(x_i)_{i \bi}$, and the shift map
 $\sigma : A^{\Z} \to A^{\Z}$ is defined
 as $(\sigma(x))_{i} = x_{i+1}$ for all $i \bi$.
Suppose that a closed subset
 $\Sigma \subset A^{\Z}$ satisfies $\sigma(\Sigma) = \Sigma$.
Then, $(\Sigma,\sigma)$ is called a \textit{(two-sided) subshift}.
A homeomorphic zero-dimensional system is expansive
 if and only if it is topologically conjugate to a two-sided subshift.
%
%
\subsection{Some notions of chaos.}
Let $(X,f)$ be a continuous surjective topological dynamical system.
\begin{defn}\label{defn:proximal}
A pair $(x,y) \in X^2$ is \textit{(forward) proximal} if
\[ \liminf_{n \to +\infty}d(f^n(x),f^n(y)) = 0. \]
\end{defn}
\begin{defn}\label{defn:Li-Yorke}
A pair $(x,y) \in X^2$ is a \textit{(forward) Li--Yorke pair}
 if $(x,y)$ is proximal and satisfies
\[ \limsup_{n \to +\infty}d(f^n(x),f^n(y)) > 0.\]
\end{defn}
\begin{defn}\label{defn:scrambled}
A subset $S \subset X$ is \textit{(forward) scrambled} if
 every $(x,y) \in S^2 \setminus \Delta_S$ is a Li--Yorke pair.
\end{defn}
\begin{defn}\label{defn:residually-scrambled}
A topological dynamical system $(X,f)$ is \textit{residually scrambled}
 if there exists a scrambled set $S \subseteq X$
 that is a dense $G_{\delta}$ subset.
\end{defn}
\begin{defn}\label{defn:completely-scrambled}
A topological dynamical system $(X,f)$ is said to be
 \textit{completely scrambled}
 if $X$ itself is a (forward) scrambled set.
\end{defn}
Because we are mainly dealing with the case in which $f$ is a homeomorphism,
 we also consider systems that are
 \textit{backward proximal, Li--Yorke, scrambled, etc.,}
 in the usual sense.
\begin{defn}\label{defn:proximal-system}
A continuous surjective topological dynamical system $(X,f)$
 is said to be \textit{proximal}
 if every pair $(x,y) \in X^2$ is (forward) proximal.
\end{defn}
Akin and Kolyada \cite{Akin_2003LiYorkeSens}
 gave a characterization of proximal systems as follows:
\begin{thm}[{\cite[Proposition 2.2]{Akin_2003LiYorkeSens}}]
\label{thm:proximal-system}
A continuous surjective topological dynamical system $(X,f)$
 is proximal
 if and only if it has a fixed point that is the unique minimal subset of X.
\end{thm}
A continuous surjective topological dynamical system $(X,f)$ is \textit{sensitive} if
 there exists $\ep > 0$ such that for all $x \in X$ and all $\delta > 0$,
 there exists $y$ with $d(x,y) < \delta$ and $n > 0$ such that
 $d(f^n(x),f^n(y)) \ge \ep$.
Akin and Kolyada
 \cite{Akin_2003LiYorkeSens} presented the notion of Li--Yorke sensitivity:
\begin{defn}\label{defn:Li-Yorke-sensitivity}
We say that a continuous surjective topological dynamical system is
 \textit{Li--Yorke sensitive} if there exists an $\ep > 0$
 such that every $x \in X$ is a limit of points $y \in X$ such that the pair
 $(x, y)$ is proximal and satisfies
\[ \limsup_{n \to +\infty}d(f^n(x),f^n(y)) \ge \ep.\]
\end{defn}
\begin{rem}\label{rem:Li-Yorke-sensitive}
Let $(X,f)$ be a two-sided subshift that is proximal.
Suppose that $X$ does not have isolated points
 and there exists a dense scrambled set $S$.
Then, $(X,f)$ is Li--Yorke sensitive.
To show this, because $(X,f)$ is symbolic,
 there exists an $\ep > 0$ such that
 for all $x,y \in X$,
 $\limsup_{n \to +\infty}d(f^n(x),f^n(y)) > 0$ implies that
 $\limsup_{n \to +\infty}d(f^n(x),f^n(y)) \ge \ep$.
Take $\ep > 0$ to satisfy the above condition.
Let $x \in X$ and $\delta > 0$ be small.
Because $S$ is dense and $X$ does not have isolated points,
 there exist $y_1,y_2 \in S$ with $y_1 \ne y_2$
 such that $d(x,y_a) < \delta$  $(a = 1,2)$.
Because $\limsup_{n \to +\infty}d(f^n(y_1),f^n(y_2)) \ge \ep$,
 for one of $a = 1,2$,
 the inequality $\limsup_{n \to +\infty}d(f^n(x),f^n(y_a)) \ge \ep/2$
follows.
\end{rem}
\subsection{Uniformly chaotic set.}
In this subsection, we introduce the notion presented by Akin \textit{et al.}
 in \cite{AKIN_2009SufficCondUnderWhichTransSysIsChaotic}.
Let $(X,f)$ be a continuous surjective topological dynamical system.
A subset $A \subset X$ is \textit{uniformly recurrent}
 if for every $\ep > 0$, there exists an arbitrarily large $k > 0$ such that
 $d(f^k(x),x) < \ep$ for all $x \in A$.
A subset $A \subset X$ is \textit{uniformly proximal}
 if $\liminf_{n \to +\infty}\diam(A) = 0$.
\begin{defn}[Akin \textit{et al.}
 \cite{AKIN_2009SufficCondUnderWhichTransSysIsChaotic}]
\label{defn:uniformly-chaotic-set}
Let $(X,f)$ be a continuous surjective topological dynamical system.
A subset $K \subset X$ is called a \textit{uniformly chaotic set} if
 there exist Cantor sets $C_1 \subset  C_2 \subset \dotsb$ such that
\enumb
\item $K = \bigcup_{i = 1}^{\infty}C_i$,
\item for each $N \bpi$, $C_N$ is uniformly recurrent, and
\item for each $N \bpi$, $C_N$ is uniformly proximal.
\enumn
Here, $(X,f)$ is \textit{(densely) uniformly chaotic} if
 $(X,f)$ has a (dense) uniformly chaotic subset.
\end{defn}
Every uniformly chaotic set is a scrambled set in the sense of Li--Yorke.
For a complete metric space without isolated points, a subset is called a
 \textit{Mycielski set} if it is a countable union of Cantor sets.
In \cite{AKIN_2009SufficCondUnderWhichTransSysIsChaotic},
 for various cases of transitive continuous surjective topological dynamical systems,
 Akin \textit{et al.} constructed uniformly scrambled sets
 as Mycielski sets both concretely and by using an instance of
 Mycielski's theorem.
For further discussion,
 see \cite{AKIN_2009SufficCondUnderWhichTransSysIsChaotic}.
We only cite the next result:
\begin{thm}
[{\cite[Corollary 3.2.~(1)]{AKIN_2009SufficCondUnderWhichTransSysIsChaotic}}]
\label{thm:transitive-fixed}
Suppose that $(X,f)$,
 a continuous surjective topological dynamical system without isolated points,
 is transitive and has a fixed point.
Then, it is densely uniformly chaotic.
\end{thm}
This implies that the topologically transitive
 proximal Cantor continuous surjective systems 
 are densely uniformly chaotic.
Because proximal Cantor systems with topological rank 2 are always
 transitive (see \cref{rem:transitive}), they are densely uniformly chaotic.
\subsection{Graph covering.}\label{sec:covering}
%
%
In many works, combinatorial studies on zero-dimensional systems
 are based on the Bratteli--Vershik representation.
Howerver, the graph covering approach of representing zero-dimensional systems   is very natural in some cases, and makes it easy to describe the phenomena.
In this subsection, we recall the method
 in \cite{Shimomura_2014SpecialHomeoApproxCantorSys}.
We also need some results in \cite{Shimomura_2014ergodic}
 to study the ergodicity of proximal Cantor systems with
 topological rank 2.
We note that, long before our construction of general graph coverings,
 Gambaudo and Martens \cite{Gambaudo_2006AlgTopMinCantSets} had 
 already introduced a combinatorial representation for all
 continuous surjective zero-dimensional minimal systems and made essential
 study on the ergodicity, which has been rather formally generalized
 to all zero-dimensional systems in \cite{Shimomura_2014ergodic}.

A pair $G = (V,E)$ consisting of a finite set $V$
 and a relation $E \subseteq V \times V$ on $V$
 can be considered as a directed graph with the set of vertices $V$
 and the set of edges $E$,
 i.e., there exists an edge from $u$ to $v$ when $(u,v) \in E$.
For a finite directed graph $G = (V,E)$, we write $V = V(G)$ and $E = E(G)$.
%
%
\begin{nota}
In this paper, we assume that a finite directed graph $G = (V,E)$
 is a surjective relation,
 i.e., for every vertex $v \in V$, there exist edges $(u_1,v),(v,u_2) \in E$.
\end{nota}
For finite directed graphs $G_i = (V_i,E_i)$ with $i = 1,2$,
 a map $\fai : V_1 \to V_2$ is said to be a \textit{graph homomorphism}
 if for every edge $(u,v) \in E_1$,
 it follows that $(\fai(u),\fai(v)) \in E_2$.
In this case, we write $\fai : G_1 \to G_2$.
For a graph homomorphism $\fai : G_1 \to G_2$,
 we say that $\fai$ is \textit{edge-surjective}
 if $\fai(E_1) = E_2$.
Suppose that a graph homomorphism $\fai : G_1 \to G_2$
 satisfies the following condition:
\[(u,v),(u,v') \in E_1 \text{ implies that } \fai(v) = \fai(v').\]
In this case, $\fai$ is said to be \textit{\pdirectional}.
Suppose that a graph homomorphism $\fai$
 satisfies both of the following conditions:
\[(u,v),(u,v') \in E_1 \text{ implies that } \fai(v) = \fai(v') \myand \]
\[(u,v),(u',v) \in E_1 \text{ implies that } \fai(u) = \fai(u').\]
Then, $\fai$ is said to be \textit{\bidirectional}.
%
%
\begin{defn}\label{defn:cover}
For finite directed graphs $G_1$ and $G_2$,
 a graph homomorphism $\fai : G_1 \to G_2$
 is called a \textit{cover} if it is a \pdirectionals \tesgh.
\end{defn}
For a sequence $G_1 \getsby{\fai_2} G_2 \getsby{\fai_3} \dotsb$
 of graph homomorphisms and $m > n$,
 we write
 $\fai_{m,n} := \fai_{n+1} \circ \fai_{n+2} \circ \dotsb \circ \fai_m$.
Then, $\fai_{m,n}$ is a graph homomorphism.
If all ${\fai_i}$ $(i \beposint)$ are edge surjective,
 then every $\fai_{m,n}$ is edge surjective.
If all ${\fai_i}$ $(i \beposint)$ are covers, every $\fai_{m,n}$ is a cover.
%
%
Let $G_0 := \left(\seb v_0 \sen, \seb (v_0,v_0) \sen \right)$
 be a singleton graph.
For a sequence of graph covers
 $G_1 \getsby{\fai_2} G_2 \getsby{\fai_3} \dotsb$, we
 attach the singleton graph $G_0$ at the head.
We refer to a sequence of graph covers
 $\covrepa{G}{\fai}$
 as a \textit{graph covering} or just a \textit{covering}.
In the original paper,
 we used the numbering $G_n \getsby{\fai_n} G_{n+1}$;
 however, in this paper, considering the numbering of Bratteli diagrams,
 we use this numbering.
%
%
Let us denote the directed graphs as $G_i = (V_i,E_i)$ for $i \benonne$.
We define the \textit{inverse limit} of $\Gcal$ as follows:
\[V_{\Gcal} := \seb (v_0,v_1,v_2,\dotsc)
 \in \prod_{i = 0}^{\infty}V_i~|~v_i = \fai_{i+1}(v_{i+1})
 \text{ for all } i \benonne \sen \text{ and}\]
\[E_{\Gcal} :=
\seb (x,y) \in V_{\Gcal} \times V_{\Gcal}~|~
(u_i,v_i) \in E_i \text{ for all } i \benonne\sen,\]
where $x = (u_0,u_1,u_2,\dotsc), y = (v_0,v_1,v_2,\dotsc) \in V_{\Gcal}$.
The set $\prod_{i = 0}^{\infty}V_i$ is equipped with the product topology.
\begin{defn}\label{defn:open-sets-of-vertices}
Let $X = V_{\Gcal}$,
 and let us define a map $f : X \to X$ by $f(x) = y$ if and only if
 $(x,y) \in E(\Gcal)$.
Then, a pair $(X,f)$ is called the \textit{inverse limit} of $\Gcal$.
\end{defn}
\begin{nota}\label{nota:vertex-and-real}
For each $n \benonne$,
 the projection from $X$ to $V_n$ is denoted by $\fai_{\infty,n}$.
For $v \in V_n$,
 we denote a closed and open set as $U(v) := \fai_{\infty,n}^{-1}(v)$.
For a subset $A \subset V_n$,
 we denote a closed and open set as $U(A) := \bigcup_{v \in A}U(v)$.
Let $\seb n_k \sen_{k \bpi}$ be a strictly increasing sequence.
Suppose that there exists a sequence
 $\seb u_{n_k} \mid u_{n_k} \in V_{n_k}, k \bpi\sen$ of vertices such that
 $\fai_{n_{k+1},n_k}(u_{n_{k+1}})= u_{n_k}$ for all $k \bpi$.
Then, there exists a unique element
 $x \in V_{\Gcal}$ such that $x \in U(u_{n_k})$
 for all $k \bpi$.
This element is denoted as $x = \lim_{k \to \infty}u_{n_k}$.
\end{nota}
%
%
%
The next theorem follows:
\begin{thm}[Theorem 3.9 and Lemma 3.5 of \cite{Shimomura_2014ergodic}]
\label{thm:0dim=covering}
Let $\Gcal : \covrepa{G}{\fai}$ be a covering.
Let $(X,f)$ be the inverse limit of $\Gcal$.
Then, $(X,f)$ is a continuous surjective zero-dimensional system.
Conversely, every continuous surjective zero-dimensional system
 can be written in this manner.
Furthermore, if $\Gcal$ is a \bidirectionals covering,
 then $(X,f)$ is a homeomorphic zero-dimensional system.
Conversely, every homeomorphic zero-dimensional system
 can be written in this manner.
\end{thm}
\begin{rem}\label{rem:partition-covering}
Let $\covrepa{G}{\fai}$ be a graph
 covering.
Let $(X,f)$ be the inverse limit.
For each $n \bni$, the set $\Ucal(G_n) := \seb U(v) \mid v \in V(G_n) \sen$
 is a closed and open partition
 such that $U(v) \cap f(U(u)) \nekuu$ if and only if $(u,v) \in E(G_n)$.
Furthermore, $\bigcup_{n \bni} \Ucal_n$ generates the topology of $X$.
Conversely, suppose that $\Ucal_n$ $(n \bni)$
 is a refining sequence of finite closed and open partitions
 of a compact metrizable zero-dimensional space $X$, $\bigcup_{n \bni}\Ucal_n$
 generates the topology of $X$, and $f : X \to X$ is a
 continuous surjective map such that for any $U \in \Ucal_{n+1}$, there exists
 $U' \in \Ucal_n$ such that $f(U) \subset U'$.
We assume that $\Ucal_0 = \seb X \sen$.
Then, we can define a graph covering such that for each $n \bni$,
 $V(G_n) = \Ucal_n$, and for every $(u,v) \in V(G_n) \times V(G_n)$,
 $(u,v) \in E(G_n)$, if and only if $f(u) \cap v \nekuu$.
It is evident that the inverse limit of this graph covering is topologically conjugate to the original.
\end{rem}
%
%
%
\begin{nota}\label{nota:graph-covering}
Let $\covrepa{G}{\fai}$ be a graph
 covering.
Let $(X,f)$ be the inverse limit.
We present a notation that is used in this paper:
\itemb
\item[(N--1)]  We write $G_{\infty} = (X,f)$;
\item[(N--2)]  We fix a metric on $X$;
\item[(N--3)]  For each $i \bni$, we write $G_i = (V_i,E_i)$
\item[(N--4)]  For each $i \bni$,
 we define $U(v) := \fai^{-1}_{\infty,i}(v)$
 for $v \in V_i$ and $\Ucal(G_i) := \seb U(v) ~|~v \in V_i\sen$; and
\item[(N--5)]  for each $i \bni$,
 there exists a bijective map $V_i \ni v \leftrightarrow U(v) \in \Ucal(G_i)$.
We note that for each $i \bni$ and $u,v \in V_i$,
 it follows that $(u,v) \in E_i$ if and only if $f(U(u)) \cap U(v) \nekuu$.
\itemn
\end{nota}
%
\begin{nota}
Let $G = (V,E)$ be a surjective directed graph.
A sequence of vertices $(v_0,v_1,\dotsc,v_l)$ of $G$
 is said to be a \textit{walk} of
 \textit{length} $l$ if $(v_i, v_{i+1}) \in E$ for all $0 \le i < l$.
We denote $w(i) := v_i$ for each $i$ with $0 \le i \le l$.
We denote $l(w) := l$.
We say that a walk $w = (v_0,v_1,\dotsc,v_l)$ is a \textit{path}
 if $v_i$ $(0 \le i \le l)$ are mutually distinct.
A walk $c = (v_0,v_1,\dotsc,v_l)$ is said to be a \textit{cycle} of period $l$
 if $v_0 = v_l$,
 and a cycle $c = (v_0,v_1,\dotsc,v_l)$
 is a \textit{circuit} of period $l$
 if $v_i$ $(0 \le i < l)$ are mutually distinct.
A circuit $c$ and a path $p$ are considered to be subgraphs of $G$
 with period $l(c)$ and length $l(p)$, respectively.
For a walk $w = (v_0,v_1,\dotsc,v_l)$,
 we define $V(w) := \seb v_i \mid 0 \le i \le l \sen$
 and $E(w) := \seb (v_i,v_{i+1}) \mid 0 \le i < l \sen$.
For a subgraph $G'$ of $G$,
 we also define $V(G')$ and $E(G')$ in the same manner.
For a walk $w$ and a subgraph $G'$,
 we also denote a closed and open set as $U(w) := \bigcup_{v \in V(w)}U(v)$
 and $U(G') := \bigcup_{v \in V(G')}U(v)$.
\end{nota}
\begin{nota}\label{nota:projected-walk}
Let $G_1 = (V_1,E_1)$ and $G_2 = (V_2,E_2)$ be finite directed graphs.
Let $\fai : G_1 \to G_2$ be a graph homomorphism.
Let $w = (v_0, v_1, \dotsc, v_l)$ be a walk in $G_1$.
Then, the walk $\fai(w) := (\fai(v_0),\fai(v_1),\dotsc,\fai(v_l))$ is defined.
\end{nota}
%
%
\begin{nota}
Let $w_1 = (u_0,u_1,\dotsc, u_l)$ and $w_2 = (v_0,v_1,\dotsc, v_{l'})$ be walks
 such that $u_l = v_0$.
Then, we denote $w_1  w_2 := (u_0,u_1,\dots,u_l,v_1,v_2,\dotsc,v_{l'})$.
Evidently, we get $l(w_1  w_2) = l+l'$.
If $c$ is a cycle of length $l$, then for any positive integer $n$,
 a cycle $c^n$ of length $ln$ is well defined.
When a walk $w  c  w'$ can be constructed, a walk $w c^0  w'$ implies that $w  w'$.
\end{nota}
\subsection{Invariant measures}\label{subsec:invariant-measures}
In \cite{Bezuglyi_2012FiniteRankBraDiagStructInvMeas},
 in the context of the finite-rank Bratteli--Vershik
 representations,
 Bezuglyi \textit{et al.}
 investigated the theory of invariant measures.
In \cite{Shimomura_2014ergodic}, we studied some relations between
 the invariant probability measures of continuous surjective zero-dimensional
 systems and the circuits of the graphs.
In this study,
 in addition to \cite{Bezuglyi_2012FiniteRankBraDiagStructInvMeas},
 we followed the first half of
 \cite[\S 3]{Gambaudo_2006AlgTopMinCantSets},
 in which Gambaudo and Martens had described the combinatorial construction of
 invariant measures.
For the calculation of ergodicity,
 we would like to use some graph covering
 methods rather than the Bratteli diagrams.
Therefore, let us summarize some results of \cite{Shimomura_2014ergodic}.
To avoid redundancy, we cite only some results and necessary tools
 from the above paper, i.e. general theories are sometimes omitted.

Let $\Gcal : \covrepa{G}{\fai}$ be a graph covering.
We write $G_{\infty} = (X,f)$ as \cref{nota:graph-covering}.
The details of our approach can be found in \cite[\S\S 4--6]{Shimomura_2014ergodic};
 hence, we omit the proofs throughout this section.
Basic notions that are not used in this manuscript are also omitted.
We refer the interested readers to the above paper.

Let us denote the set of all invariant Borel finite measures in $(X,f)$
 as $\sM_f = \sM_f(X)$,
 and the invariant Borel probability measures as $\sP_f = \sP_f(X)$.
Both $\sM_f$ and $\sP_f$ are endowed with a weak topology.

For each $e = (u,v) \in E_n$,
 we denote a closed and open set as $U(e) := U(u) \cap f^{-1}(U(v))$.
Note that for $e \neq e'~(e,e' \in E_n)$,
 it follows that $U(e) \cap U(e') = \kuu$.
%
We note that if we consider the sets $\seb U(v) \mid v \in V_n, n \bni \sen
 = \bigcup_ {n \bni} \Ucal(G_n)$
 and $\seb U(e) \mid e \in E_n, n \bni \sen$,
 then they are open bases of the topology of $X$.
Suppose that non-negative real values $\mu(U)$ are assigned for all $U \in \sO_V$ and $U \in \sO_E$ satisfying the condition in \cite[Lemma 4.1]{Shimomura_2014ergodic}.
Then, $\mu$ is a countably additive measure on the set of all open sets of $X$.
Therefore, as is well known, $\mu$ can be extended to the unique Borel measure on $X$.
A Borel measure $\mu$ on $X$, constructed in this manner, is a probability measure if and only if $1 = \sum_{v \in V_n} \mu(U(v))$ for all $n \bni$.
Thus, we obtain the following:
\begin{prop}\label{lem:creation-of-measures}
Suppose that non-negative real values $\mu(U)$ are assigned for all $U \in \sO_E$ such that the conditions
 in \cite[Lemma 4.1]{Shimomura_2014ergodic} are satisfied.
Then, the resulting $\mu$ is an invariant measure.
\end{prop}
Let $G = (V,E)$ be a graph.
We denote the set of all circuits (subgraphs) as $\sC(G)$.
Let $c \in \sC(G)$ and $s \in \Real$.
Then, we denote $sc := \sum_{e \in E(c)}se$.
For $c \in \sC(G)$, we denote $\tilc = (1/l(c)) \cdot c$.
Each linear combinataion $\sum_{c \in \sC(G)}s_c  \tilc$
 with $\sum_{c \in \sC(G)}s_c = 1$ satisfies the condition
 in \cite[Lemma 4.1]{Shimomura_2014ergodic}.
Conversely, if a linear combination $\sum_{e \in E}s_e e$ satisfies
 the condition in \cite[Lemma 4.1]{Shimomura_2014ergodic},
 then there exists a linear combination $\sum_{c \in \sC(G)}s_c  \tilc$
 with $\sum_{c \in \sC(G)}s_c = 1$ that generates $\sum_{e \in E}s_e e$,
 for the proof, see \cite[Proposition 4.7]{Shimomura_2014ergodic}.
The set of linear combinations is denoted 
 $\Pcal(G) := \seb \sum_{c \in \sC(G)}s_c  \tilc \mid
 \sum_{c \in \sC(G)}s_c = 1\sen$.
Because $\fai_{m,n} :G_m \to G_n$ is a graph cover for each $m > n$,
 there exists a natual linear map $(\fai_{m,n})_* : \Pcal(G_m) \to \Pcal(G_n)$.
\begin{nota}\label{nota:deltan}
For each $n \bni$, we denote $\sC_n := \sC(G_n)$ and $\Delta_n := \Pcal(G_n)$.
For each $m > n$, we define an affine map as $\xi_{m,n} := (\fai_{m,n})_*|_{\Delta_m} : \Delta_m \to \Delta_n$.
For each $n \bni$ and $c \in \sC_n$, it follows that $\tilc \in \Delta_n$.
\end{nota}
\begin{nota}
We denote an inverse limit as follows:
\[\Deltainf := \seb (x_0,x_1,\dotsc) ~|~\myforall i \bni,~ x_i \in \Delta_i,~\text{and for all } m > n,~ \xi_{m,n}(x_m) = x_n \sen.\]
Further, $\Deltainf$ is endowed with product topology.
For each $n \bni$, we denote the projection as $\xi_{\infty,n} : \Deltainf \to \Delta_n$.
It is evident that $\xi_{\infty,n}$ is continuous.
For $\infty \ge k > j > i$, it follows that $\xi_{j,i} \circ \xi_{k,j} = \xi_{k,i}$.
\end{nota}
It is well known that $\sP_f$ is a compact, metrizable space.
By definition, $\Deltainf$ is also a compact, metrizable space.
%
%
%
\begin{nota}\label{nota:measure-inverse-limit-of-circuits}
Let $\mu \in \sP_f$.
We write $\bar{\mu} := \pstrzinf{\mu}$, where $\mu_n = \sum_{e \in E_n}\mu(U(e))e$ for all $n \bni$.
Then, it follows that $\bar{\mu} \in \Deltainf$ and $\mu_n \in \Delta_n$ for each $n \bni$.
Following \cite{Gambaudo_2006AlgTopMinCantSets},
 we denote the map $\mu \mapsto \bar{\mu}$ as $p_*$.
Thus, we have the map $p_* : \sP_f \to \Deltainf$.
Note that we can write $\bar{\mu} = \pstrzinf{\mu}$,
 where $\mu_n = \sum_{c \in \sC_n}s_c \tilc$.
\end{nota}
We have checked Proposition 3.2 of \cite{Gambaudo_2006AlgTopMinCantSets}
 for our case.
\begin{prop}\label{prop:description-of-measures-by-covers}
The map $p_* : \sP_f \to \Deltainf$ is an isomorphism.
\end{prop}
\begin{nota}
Let $\seb x_{n_i} \sen_{i \bpi}$ be a sequence such that $n_1 < n_2 < \dotsb$, and $x_{n_i} \in \Delta_{n_i}$ for all $i \bpi$.
Suppose that, for all $m \bni$, $\lim_{i \to \infty}\xi_{n_i,m}(x_{n_i}) = y_m \in \Delta_m$.
Then, we denote $\lim_{i \to \infty}x_{n_i} := (y_0,y_1,\dotsc) \in \Deltainf$.
\end{nota}
Hereafter, we consider restrictions of the entire set $\Nonne$ on infinite subsets of $\Nonne$.
Let $\bN \subseteq \Nonne$ be an infinite subset.
Because the restricted sequence $\seb G_n \sen_{n \in \bN}$ of $\Gcal$ produces the same inverse limit $G_{\infty}$, it is convenient to consider such restrictions.
This change can be done for some calculations of the invariant measures
 without telescoping $\Gcal$ itself.
In \cite{Shimomura_2014ergodic}, we made a general theory taking sufficietly
 many circuits $\seb \Ccal_n \subseteq \sC_n \sen_{n \in \bN}$
 for sufficiently many $\bN \subseteq \Nonne$
 to express particular measures for certain aim.
However, in this paper, we need not this general argument.
Thus, we consider $\bN = \Nonne$ and $\Ccal_n = \sC_n$ ($n \bni$).
In the same way, we define
\[\sC_{\infty} := \seb \seb c_{n} \sen_{n \bni} \mid c_n \in \sC_n
 \myforall n \bni, \lim_{n \to \infty}\tilc_{n} \text{ exists} \sen. \]
For each $x = \seb c_n \sen_{n \bni} \in \sC_{\infty}$,
 we denote $\bar{x} = \lim_{n \to \infty}\tilc_{n} \in \Deltainf$.
Further, we write $\bar{\sC}_{\infty} := \seb \bar{x}~|~x \in \sC_{\infty}\sen$.
Note that because every $\Delta_n$ ($n \bni$) is compact,
 it is obvious that,
 for each sequence of circuits $\seb c_{n} \sen_{n \bni}$
 with $c_n \in \sC_n$ for all $n \bni$,
 there exists a subsequence $\seb n_i \sen_{i \bpi}$
 such that there exists a $\lim_{i \to \infty}\tilc_{n_i}$.
Note that this does not directly imply $\sC_{\infty} \nekuu$.

As shown by the next theorem, each ergodic measure is expressed as a limit of
 a sequence of circuits.
\begin{thm}\label{thm:ergodic-measures-are-from-essential-circuits}
Let $\mu \in \sP_f$ be an ergodic measure.
Let $\bN \subseteq \Nonne$ be an infinite subset.
Suppose that $\Ccal$ is a system of circuits that expresses $\mu$.
Then, it follows that $\bar{\mu} \in \ovCcalinf$.
\end{thm}
We note that, by this theorem,
 we need not take further subsequences of $\bN$ to get the convergent sequences.
In \cite[Proposition 3.3]{Gambaudo_2006AlgTopMinCantSets}, Gambaudo and Martens showed that if the number of loops in $X_n$ is uniformly bounded by $k$, then $f$ has at most $k$ ergodic invariant probability measures. They reported their findings in the context of Cantor minimal continuous surjection. Nevertheless, the proof is still valid in our case of zero-dimensional proximal systems.
We have checked this fact in the context of this section:
\begin{thm}[{\cite[Theorem 6.2]{Shimomura_2014ergodic}}]\label{thm:finite-ergodicity}
Let $\bN$ be an infinite subset of $\Nonne$.
Suppose that $\abs{\sC_n} \le k$ for all $n \in \bN$.
Then, $\sP_f$ has at most $k$ ergodic measures.
\end{thm}
\section{Bratteli--Vershik representations}
\label{sec:Bratteli-Vershik-representations}
In the previous section, we presented a combinatorial approach for representing 
 every continuous surjective zero-dimensional system by our graph covering.
\begin{nota}
By $(V,E)$, we refer to not only a finite directed graph but also a Bratteli diagram.
To avoid ambiguity,
 for a finite directed graph $G = (V,E)$, we write $V = V(G)$ and $E = E(G)$,
 and the Bratteli diagram is simply expressed as $(V,E)$.
\end{nota}
From this section, we return to homeomorphic zero-dimensional systems.
Long before our construction of graph coverings,
 Herman, Putnam, and Skau \cite{HERMAN_1992OrdBratteliDiagDimGroupTopDyn}
 had shown that
 every Cantor essentially minimal homeomorphism is represented by
 an essentially simple ordered Bratteli--Vershik model.
In \cite[\S 1]{GJERDE_2000BratteliVershikModelCantorMinSysAppToeplFlows}, Gjerde and Johansen described
 how Bratteli diagrams are employed to obtain models for
 Cantor minimal homeomorphisms.
Nevertheless, because we treat proximal systems that are not minimal,
 we also have to return to the original
 \cite{HERMAN_1992OrdBratteliDiagDimGroupTopDyn}.
We follow Gjerde and Johansen
 \cite[\S 1]{GJERDE_2000BratteliVershikModelCantorMinSysAppToeplFlows}
 to describe the Bratteli--Vershik models for zero-dimensional
 essentially minimal systems.
\begin{defn}
A \textit{Bratteli diagram} is an infinite directed graph $(V,E)$, where $V$ is the vertex
set and $E$ is the edge set.
These sets are partitioned into non-empty disjoint finite sets 
$V = V_0 \cup V_1 \cup V_2 \cup \dotsb$ and $E = E_1 \cup E_2 \cup \dotsb$,
 where $V_0 = \seb v_0 \sen$ is a one-point set.
Each $E_n$ is a set of edges from $V_{n-1}$ to $V_n$.
Therefore, there exist two maps $r,s : E \to V$, such that $r:E_n \to V_n$
 and $s : E_n \to V_{n-1}$ for $n \bpi$, i.e., the range map and the source map respectively.
Moreover, the conditions $s^{-1}(v) \nekuu$ for all $v \in V$ and
$r^{-1}(v) \nekuu$ for all $v \in V \setminus V_0$ are assumed.
We say that $u \in V_{n-1}$ is connected to $v \in V_{n}$ if there
 exists an edge $e \in E_n$ such that $s(e) = u$ and $r(e) = v$.
Unlike the case of graph coverings,
 multiple edges between $u$ and $v$ are permitted.
The \textit{rank $K$} of a Bratteli diagram is defined as
 $K := \liminf_{n \to \infty}\num{V_n}$,
 where $\num{V_n}$ is the number of elements in $V_n$.
\end{defn}
\begin{nota}
In this section, we also consider a type of directed graph and its covering maps.
Directed graphs are denoted as $G_n$ $(n \bni)$, while the sets of vertices
 are denoted as $V(G_n)$ and the sets of edges are denoted as $E(G_n)$.
In this section,
 we use the notation $V_n$ and $E_n$ only for Bratteli diagrams.
\end{nota}
%
%
Let $(V,E)$ be a Bratteli diagram and $m < n$ be non-negative integers.
We define
\[E_{m,n} := 
 \seb p \mid p \text{ is a path from a } u \in V_m \text{ to a } v \in V_n \sen.\]
Then, we can construct a new Bratteli diagram $(V',E')$ as follows:
\[ V' := V_0 \cup V_1 \cup \dotsb \cup V_m \cup V_n \cup V_{n+1} \cup \dotsb \]
\[ E' := E_1 \cup E_2 \cup \dotsb \cup E_m \cup E_{m,n} \cup E_{n+1} \cup \dotsb. \]
The source map and the range map are also defined naturally.
This procedure is called {\it telescoping}.

\begin{defn}
Let $(V,E)$ be a Bratteli diagram such that
$V = V_0 \cup V_1 \cup V_2 \cup \dotsb$ and $E = E_1 \cup E_2 \cup \dotsb$
 are the partitions,
 where $V_0 = \seb v_0 \sen$ is a one-point set.
Let $r,s : E \to V$ be the range map and the source map, respectively.
We say that $(V,E,\le)$ is an {\it ordered}\/ Bratteli diagram if
 a partial order $\le$ is defined on $E$ such that 
 $e, e' \in E$ is comparable if and only if $r(e) = r(e')$.
 Thus, we have a linear order on each set $r^{-1}(v)$ with $v \in \Vp$.
The edges $r^{-1}(v)$ are numbered from $1$ to $\num{r^{-1}(v)}$.
\end{defn}
Let $n > 0$ and $e = (e_n,e_{n+1},e_{n+2},\dotsc), e'=(e'_n,e'_{n+1},e'_{n+2},\dotsc)$ be cofinal paths from the vertices of $V_{n-1}$, which might be different.
We get the lexicographic order $e < e'$ as follows:
\[\myif k \ge n \text{ is the largest number such that } e_k \ne e'_k, \mythen e_k < e'_k.\]%
\begin{defn}
Let $(V,E,\le)$ be an ordered Bratteli diagram.
Let $E_{\max}$ and $E_{\min}$ denote the sets of maximal and minimal
edges, respectively.
An infinite path is maximal (minimal) if all the edges constituting the
path are elements of $E_{\max}$ ($E_{\min}$).
\end{defn}
We provide the following notation
 that does not appear in
 \cite{GJERDE_2000BratteliVershikModelCantorMinSysAppToeplFlows}.
\begin{nota}
Let $v \in V\setminus V_0$.
Because $r^{-1}(v)$ is linearly ordered, 
 we denote the maximal edge $e(v,\max) \in r^{-1}(v)$
 and the minimal edge $e(v,\min) \in r^{-1}(v)$.
\end{nota}

\begin{defn}
As in \cite{HERMAN_1992OrdBratteliDiagDimGroupTopDyn},
 an ordered Bratteli diagram $(V,E)$ is said to be
 \textit{essentially simple}
 if there exists a unique infinite path
 $p_{\max} = (e_{1,\max},e_{2,\max},\dotsc)$
 with $e_{i,\max} \in E_{\max} \cap E_i$ for all $i \bpi$,
 and there exists a unique infinite path
 $p_{\min} = (e_{1,\min},e_{2,\min},\dotsc)$
 with $e_{i,\min} \in E_{\min} \cap E_i$ for all $i \bpi$.
\end{defn}

\begin{defn}[Vershik map]
Let $(V,E,\le)$ be an essentially simple ordered Bratteli diagram
 with the unique maximal path $p_{\max}$ and the unique minimal path
 $p_{\min}$.
Let 
\[E_{0,\infty} := \seb (e_1,e_2,\dotsc) \mid r(e_i) = s(e_{i+1})
 \myforall i \ge 1 \sen,\]
with the subspace topology of the product space $\prod_{i = 1}^{\infty}E_i$.
We can define the {\it Vershik map}\/ $\phi : E_{0,\infty} \to E_{0,\infty}$ as follows:

\noindent If $e = (e_1,e_2,\dotsc) \ne p_{\max}$,
 then there exists the least $n \ge 1$ such that 
 $e_n$ is not maximal in $r^{-1}(r(e_n))$.
Then, we can select the least $f_n > e_n$ in $r^{-1}(r(e_n))$.
Let $v_{n-1} = s(f_n)$.
Then, it is easy to get the unique least path $(f_1,f_2,\dotsc,f_{n-1})$
 from $v_0$ to $v_{n-1}$.
Now, we can define
\[\phi(e) := (f_1,f_2,\dotsc,f_{n-1},f_n,e_{n+1},e_{n+2},\dotsc)\],
and separately, we define $\phi(p_{\max}) = p_{\min}$.
The map $\phi : E_{0,\infty} \to E_{0,\infty}$ is called the {\it Vershik map}.
The system $(E_{0,\infty},\phi)$ is a homeomorphic zero-dimensional topological
 dynamical system (see \cite{HERMAN_1992OrdBratteliDiagDimGroupTopDyn}).
The system $(E_{0,\infty},\phi)$ is called
 the \textit{essentially simple Bratteli--Vershik model}
 determined by $(V,E,\le)$,
 and $(X,f)$ \textit{has essentially simple Bratteli--Vershik model}
 if $(X,f)$ is topologically conjugate to $(E_{0,\infty},\phi)$.
\end{defn}
The next theorem is a part of
 \cite[Theorem 4.6]{HERMAN_1992OrdBratteliDiagDimGroupTopDyn}.
\begin{thm}[{\cite[Theorem 4.6]{HERMAN_1992OrdBratteliDiagDimGroupTopDyn}}]
\label{thm:essentially-simple-representation}
The essentially simple Bratteli--Vershik model
 is a homeomorphic essentially minimal zero-dimensional system.
Conversely, a homeomorphic essentially minimal zero-dimensional system has 
 an essentially simple Bratteli--Vershik model.
In the Bratteli--Vershik model, both $p_{\max}$ and $p_{\min}$ are
 points in the unique minimal set.
\end{thm}
In \cite{DOWNAROWICZ_2008FiniteRankBratteliVershikDiagAreExpansive},
 Downarowicz and Maass introduced the topological rank for
 a homeomorphic Cantor minimal system.
We define the same for (homeomorphic) zero-dimensional essentially minimal systems as follows:
\begin{defn}[Topological Rank]\label{defn:rank}
Let $(X,f)$ be a homeomorphic essentially minimal zero-dimensional system.
Then, \textit{the topological rank of $(X,f)$} is $1 \le K \le \infty$
 if it has an essentially simple Bratteli--Vershik model
 determined by an essentially simple ordered Bratteli diagram with rank $K$,
 and $K$ is the minimum of such numbers.
\end{defn}%
%
%
\subsection{Graph coverings of Kakutani--Rohlin type}
We now characterize the homeomorphic zero-dimensional essentially minimal
 systems by means of graph coverings.
\begin{defn}\label{defn:gf8}
A finite directed graph $G$ is called a \textit{generalized figure 8} if
 there exist \textit{center} $v_0 \in V(G)$ and distinct
 circuits $c_1,c_2,\dotsc,c_r$ with periods $p(t)$ $(1 \le t \le r)$.
The circuits are denoted as
 $c_t = (v_{t,0} = v_0, v_{t,1}, v_{t,2}, \dotsc, v_{t,p(t)} = v_0)$
 for $1 \le t \le r$, and they satisfy the following requirements:
\itemb
\item $E(G) = \bigcup_{1 \le t \le r}E(c_t)$, and
\item $V(c_s) \cap V(c_t) = \seb v_0 \sen$ for each $s \ne t$.
\itemn
The term ``generalized figure 8'' is abbreviated as \textit{GF8}.
In the case of GF8, every $c \in \sC(G)$ is considered as a walk from
 the center.
Further, $r$ denotes the \textit{rank} of $G$.
Note that the number of vertices of $G$ is counted as 
 $1+ \sum_{1 \le t \le r}(p(t)-1)$.
\end{defn}
Let $G_0$ be the singleton graph that corresponds to
 the top vertex of the Bratteli diagram.
We shall construct each $G_n$ $(n \ge 1)$ as GF8 with the center
 $v_{n,0}$ and distinct circuits
 $c_{n,1}, c_{n,2},\dotsc, c_{n,r_n}$.
We express the period of each circuit $c_{n,t}$ with $1 \le t \le r_n$
 as $p(n,t) \ge 1$.
Thus, we write
\[c_{n,t}
 = (v_{n,t,0} = v_{n,0}, v_{n,t,1}, v_{n,t,2}, \dotsc, v_{n,t,p(n,t)}
 = v_{n,0})\]
for $n \ge 1$ and $1 \le t \le r_n$.
\begin{defn}\label{defn:KR-cov}
We say that a covering
 $\covrepa{G}{\fai}$ is of
 \textit{Kakutani--Rohlin type} if $G_n$ is a GF8 and
 $\fai_{n+1}(v_{n+1,0}) = v_{n,0}$ for every $n \bni$.
We abbreviate a covering of Kakutani--Rohlin type as a \textit{KR-covering}.
A KR-covering has \textit{rank} $1 \le K \le \infty$
 if $\liminf_{n \to \infty}r_n = K$.
Note that for each $n \bni$ and $t$ with $1 \le t \le r_{n+1}$,
 we can write
\[\fai_{n+1}(c_{n+1,t}) = c_{n,t,1} c_{n,t,2} \dotsb c_{n,t,w(n,t)},\]
 where $c_{n,t,j}$'s are circuits of $G_n$,
 i.e., $c_{n,t,j} = c_{n,i(n,t,j)}$ for some
 $1 \le i(n,t,j) \le r_n$.
Furthermore, because of the \pdirectionalitys of a covering, it follows that
 $c_{n,t,1}$'s are independent of $t$.
\begin{nota}
Hereafter, we assume that $c_{n,t,1} = c_{n,1}$ for all $n,t$, i.e.,
 every upper circuit winds the lower circuits with $c_{n,1}$ first.
\end{nota}
We note that for an $n \bni$, $\fai_n$ is \bidirectionals if and only if
 $c_{n,t,w(n,t)}$'s are independent of $t$.
\end{defn}
%
%
%
%
\begin{defn}\label{defn:KR-covering-model}
If $\covrepa{G}{\fai}$ be a KR-covering,
 then the inverse limit $G_{\infty} = (X,f)$
 is called a \textit{KR-covering model} determined by
 the KR-covering.
If a continuous surjective zero-dimensional system $(X,f)$
 is topologicallly conjugate to $G_{\infty}$ for some
 KR-covering $\covrepa{G}{\fai}$, then we say that 
 $(X,f)$ \textit{has a KR-covering model}.
\end{defn}

\begin{rem}
As in the case of essentially simple Bratteli--Vershik models,
 some topological rank can be defined by the usage of 
 ranks in KR-covering models.
Later, in \cref{thm:coincidence-of-ranks-essentially-simple-BV-and-KR-cov},
 we find that the two topological ranks coincide
 if we consider homeomorphic essentially minimal zero-dimensional
 systems.
\end{rem}

\begin{rem}\label{lem:exist-kr-partition}
Let $\covrepa{G}{\fai}$ be a \bidirectionals KR-covering.
We are able to 
 consider a refining sequence of Kakutani--Rohlin partitions
 (see \cite[Lemma 4.1]{HERMAN_1992OrdBratteliDiagDimGroupTopDyn}).
For each $n \ge 0$ and $c_{n,t}$ with $1 \le t \le r_n$, 
 let
 $c_{n,t} = (v_{n,t,0}=v_{n,0},v_{n,t,1}, \dotsc, v_{n,t,{p(n,t)}} = v_{n,0})$.
We write $B_t = f^{-1}(U(v_{n,t,1})) \subset U(v_{n,0})$.
Then, $f^{k}(B_t) = U(v_{n,t,k})$ for $1 \le k < p(n,t)$.
We write $\xi_t = \seb f^k(B_t) \mid 0 \le k < p(n,t) \sen$.
The zero-dimensional system $(X,f)$ is partitioned into
 Kakutani--Rohlin partition $\Xi_n :=\bigcup_{1 \le t \le r_n} \xi_t$.
It is clear that $\Xi_{n+1}$ refines $\Xi_n$ for all $n \ge 0$.
\end{rem}
%
%
%
%
%
Hereafter, in this section, we show that
 a \bidirectionals KR-covering is linked with
 an essentially simple ordered Bratteli diagram.
First, we construct an essentially simple ordered Bratteli diagram
 from a \bidirectionals KR-covering.
\begin{lem}\label{lem:KR-covering->Bratteli-diagram}
Let
 $\covrepa{G}{\fai}$ be a
 \bidirectionals KR-covering.
An essentially simple ordered Bratteli diagram $(V,E)$
 is constructed as follows:
 $V_n = \sC(G_n)$ $(n \bni)$; if for $n \bni$,
\[\fai_{n+1}(c_{n+1,t})
 = c_{n,t,1} c_{n,t,2} \dotsb c_{n,t,w(n,t)}
 \text{ for each } 1 \le t \le r_{n+1},\]
 an edge that belongs to $E_{n+1} \subset E$
 is made from each $c_{n,t,j}$ to $c_{n+1,t}$
 for $1 \le j \le w(n,t)$ that is numbered by $j$.
In particular, the ranks do not differ.
The Bratteli--Vershik system is topologically conjugate to $G_{\infty}$.
Moreover, it follows that the ranks coincide.
\end{lem}
\begin{proof}
To check that $(V,E)$ is an essentially simple ordered Bratteli diagram, we have to show that the maximal infinite path from $V_0$ is unique and
 the minimal infinite path from $V_0$ is unique.
Let $n \bpi$.
For each $v \in V_n$ of the Bratteli diagram, there is a corresponding
circuit $c_{n,t}$ of $G_n$.
Then, $c_{n-1,t,1}$ is independent of $t$.
Let $v_{n-1,1}$ be the corresponding vertex of $V_{n-1} \subset V$ of the 
 Bratteli diagram, i.e.,
 there exists a minimal edge $e \in E_n$ from $v_{n-1,1}$ to $v$.
Let $e_{n}$ be the minimal edge from $v_{n-1,1}$ to $v_{n,1}$.
Then, $(e_{1},e_{2},\dotsc)$ is the unique infinite minimal path
 from $V_0$.
Because of the \bidirectionalitys condition, the existence and uniqueness
 of the infinite
 maximal path from $V_0$ also follow by the same argument.
It is straightforward to check the remaining conditions of the ordered Bratteli diagrams.
To check the last statement, let $(X,f) = G_{\infty}$.
Let $x_0 := (v_{0,0},v_{1,0},v_{2,0},\dotsc) \in X$, where $v_{n,0}$ is the center of $G_n$.
Let $x = (v_0,v_1,v_2,\dotsc) \in X \setminus \seb x_0 \sen$.
Let $N$ be the least positive integer such that for all $n \ge N$,
 $v_n \ne v_{n,0}$.
Let $n \ge N$.
Let $c_n \in \sC(G_n)$ be a unique circuit such that $v_n \in V(c_n)$.
We have considered that $c_n \in V_n \subset V$, a vertex of the Bratteli diagram.
Observe that, because $v_{n+1} \in V(c_{n+1})$ is already determined,
 by the cover $\fai_{n+1} : G_{n+1} \to G_n$, one can identify the unique edge $e_{n+1}$ from $c_n$ to $c_{n+1}$.
Because $v_N \in V(c_N)$ is determined, by the cover $\fai_N : G_N \to G_{N-1}$, the unique $c_{N-1} \in V_{N-1}$ and the unique edge $e_N \in E_N$ from $c_{N-1}$ to $c_N$ that may not be minimal is determined.
For all $0 \le i < N-1$, we choose $c_i := c_{i, 1}$ and the edge $e_{i+1}$ from $c_{i}$ to $c_{i+1}$ is the minimal edge.
Thus, we have constructed $h(x) := (e_1,e_2,\dotsc) \in E_{0,\infty}$.
We denote $h(x_0) := p_{\min} \in E_{0,\infty}$.
Recall that we have denoted for all $n \bni$ and $1 \le t \le r_{n+1}$,
\[\fai_{n+1}(c_{n+1,t}) = c_{n,t,1} c_{n,t,2} \dotsb c_{n,t,w(n,t)}.\]
By the \bidirectionalitys condition, $c_{n,t,w(n,t)}$ is independent of $t$.
We denote this as $c_{n,\max}$.
There exists a maximal edge $e_{n,\max}$ from $c_{n-1,\max}$ to $c_{n,\max}$ for all $n \bpi$.
Let $v_{n,\max} := c_{n,\max}(l(c_{n,\max})-1) \in V(G_n)$ for all $n \bpi$.
Then, it follows that we can define
 $x_1 : = (v_0,v_{1,\max},v_{2,\max},\dotsc) \in X$.
It is obvious that $h(x_1) = p_{\max}$.
We leave it to the readers to check that $h : X \to E_{0,\infty}$
 is a homeomorphism and $h \circ f = \phi \circ h$.
\end{proof}

We now show that the converse is true, i.e.,
 we construct a \bidirectionals KR-covering from an essentially simple
 ordered Bratteli diagram.
For this purpose, we need a few lemmas and a definition.
The readers might skip the following detailed arguments.
\begin{lem}\label{lem:maximal-path-start-from:essentially-minimal}
Suppose that $(V,E)$ is an essentially simple ordered Bratteli diagram.
Let $p_{\max} = (e_{1,\max},e_{2,\max},\dotsc)$
 be the infinite maximal path from $V_0$.
Let $v_{i,\max} = r(e_{i,\max})$ $(i \bpi)$.
Let $n \bpi$.
Then, there exists an $n_0 > n$ such that if
 $(e_{n+1,\max},e_{n+2,\max},\dotsc,e_{m,\max})$
 with $m \ge n_0$ is maximal, then
 $s(e_{n+1,\max}) = v_{n,\max}$.
\end{lem}
\begin{proof}
Suppose that the claim fails.
Then, there exists an infinite sequence $n < m(1) < m(2) < \dotsb$ of integers
 and corresponding sequence $p_k = (e_{n+1,k},e_{n+2,k},\dotsc,e_{m(k),k})$
 $(k \bpi)$
 of maximal paths such that $s(e_{n+1,k}) \ne v_{n,\max}$.
Taking a subsequence if necessary,
 we can assume that $e_{n+1,k}$ is independent
 of $k$.
Thus, we get an infinite maximal path $(e'_{n+1},e'_{n+2},\dotsc)$
 such that $s(e'_{n+1}) \ne v_{n,\max}$.
It is easy to get a maximal path from $V_0$ to $v_{n,0}$.
Connecting these paths, we get an infinite maximal path from $V_0$ that is not
 $p_{\max}$, which is a contradiction.
\end{proof}
An identical argument shows the following:
\begin{lem}\label{lem:minimal-path-start-from:essentially-minimal}
Suppose that $(V,E)$ is an essentially simple ordered Bratteli diagram.
Let $p_{\min} = (e_{1,\min},e_{2,\min},\dotsc)$
 be the infinite minimal path from $V_0$.
Let $v_{i,\min} = r(e_{i,\min})$ $(i \bpi)$.
Let $n \bpi$.
Then, there exists an $n_0 > n$ such that if
 $(e_{n+1},e_{n+2},\dotsc,e_m)$ with $m \ge n_0$ is minimal, then
 $s(e_{n+1}) = v_{n,\min}$.
\end{lem}
\begin{defn}
Let $(V,E)$ be an essentially simple ordered Bratteli diagram.
We have written the maximal path as
 $p_{\max} = (e_{1,\max},e_{2,\max},\dotsc)$
 with $e_{i,\max} \in E_{\max} \cap E_i$ for all $i \bpi$,
 and the minimal path as 
 $p_{\min} = (e_{1,\min},e_{2,\min},\dotsc)$
 with $e_{i,\min} \in E_{\min} \cap E_i$ for all $i \bpi$.
Let $v_{n,\max} := r(e_{n,\max})$ for all $n \bni$,
 and $v_{n,\min} := r(e_{n,\min})$ for all $n \bni$.
Applying
 \cref{lem:maximal-path-start-from:essentially-minimal,lem:minimal-path-start-from:essentially-minimal},
 and telescoping,
 we get an essentially simple ordered Bratteli diagram
 with the additional property such that
\itemb
\item for all $m \bpi$ and $v \in V_m$, it follows that 
 $v_{m-1,\max} = s(e(v,\max))$ and
\item for all $m \bpi$ and $v \in V_m$, it follows that 
 $v_{m-1,\min} = s(e(v,\min))$.
\itemn
We say that an essentially minimal ordered Bratteli diagram has
 \textit{reduced form}
 if it satisfies the above condition.
\end{defn}
Hereafter, we assume that an essentially simple ordered Bratteli diagram
 $(V,E)$ has reduced form.
\begin{nota}
For each $n \bpi$ and $v \in V_n$, we denote 
 $P(v) := 
 \seb p = (e_1,e_2,\dotsc,e_n) \mid
 p \text{ is a path from } V_0 \text{ to } v \sen$.
Then, by the lexicographic order, $P(v)$ is linearly ordered.
We order $P(v) = \seb p_{v,1},p_{v,2},\dotsc p_{v,w(v)} \sen$ with
 $p_{v,i} < p_{v,j}$ if and only if $i < j$.
We define for each $p \in P(v)$,
 $[p] := \seb q \in E_{0,\infty} \mid
 \text{ the first } n \text{ edges of } q \text{ coincide with } p \sen$.
We denote $P_n := \bigcup_{v \in V_n} P(v)$.
For each $n \bpi$, we denote by $p_{n,\min}$ ($p_{n,\max}$)
 the finite path that consists of the first $n$ edges
 of $p_{\min}$ ($p_{\max}$).
For each $n \bpi$, we denote $D_{n,0} := \bigcup_{v \in V_n}[p_{v,1}]$.
For each $n \bpi$, each $v \in V_n$, and $0 < i < w(v)$,
 we denote $D_{n,v,i} := [p_{v,i+1}]$.
For $n \bpi$, we get a partition by closed and open sets
\[\Dcal_n := D_{n,0} \cup \seb D_{n,v,i} \mid v \in V_n, 1 \le i < w(v) \sen.\]
We denote $\Dcal_0 := \seb E_{0,\infty} \sen$.
\end{nota}
\begin{lem}\label{lem:first-n-coincide}
Let $n \bpi$ and $v \in V_{n+1}$.
The path that consists of the first $n$ edges of $p_{v,1}$
 coincides with $p_{n,\min}$ ($p_{n,\max}$).
\end{lem}
\begin{proof}
Because we have assumed the reduced form,
 the proof is straightforward.
\end{proof}
\begin{lem}\label{lem:bratteli-to-covering}
We get the following:
\enumb
\item\label{item:bratteli-to-cov:a} for $n \bpi$ and $v \in V_n$,
 $\phi([p_{v,i}]) = [p_{v,i+1}]$ for all $1 \le i < w(v)$,
\item\label{item:bratteli-to-cov:b} for $n \bpi$,
 $\bigcup_{v \in V_n}\phi([p_{v,w(v)}]) = \bigcup_{v \in V_n}[p_{v,1}]$,
 and
\item\label{item:bratteli-to-cov:c} for $n \bpi$,
 $D_{n+1,0} \subset [p_{n,\min}]$, and
\item\label{item:bratteli-to-cov:d} $\bigcup_{n \bni} \Dcal_n$
 generates the topology of $E_{0,\infty}$.
\enumn
\end{lem}
\begin{proof}
For a fixed $n \bpi$, $\seb [p] \mid p \in P_n \sen$ is a partition 
 by the closed and open sets of $E_{0,\infty}$.
By the definition of the Vershik map,
 \cref{item:bratteli-to-cov:a} follows.
We recall that $\phi$ is a homeomorphism.
Thus, \cref{item:bratteli-to-cov:b} follows.
Further, \cref{item:bratteli-to-cov:c} follows from \cref{lem:first-n-coincide}.
Evidently, \cref{item:bratteli-to-cov:c} implies \cref{item:bratteli-to-cov:d}.
\end{proof}
%
%
For each $n \bni$, we construct a finite directed graph $G_n$ such that
 $V(G_n) := \Dcal_n$, and $(u,v) \in E(G_n)$
 if and only if $\phi(u) \cap v \nekuu$.
Thus, in $G_n$, we get circuits
\[(D_{n,0},D_{n,v,1},D_{n,v,2},\dotsc,D_{n,v,w(v) -1},D_{n,0})~~
 \text{ for each } v \in V_n.\]
In particular, the rank of $G_n$ is equal to $\abs{V_n}$.
We remark that all vertices $v \in V_n$ with $w(v) = 1$ are merged
 into the circuit $(D_{n,0},D_{n,0})$.
Other circuits are disjoint except at $D_{n,0}$.
For each $p = (e_1,e_2,\dotsc,e_{n+1}) \in P_{n+1}$, evidently,
 $[p] \subseteq [(e_1,e_2,\dotsc,e_n)]$.
We also note that,
 by \cref{item:bratteli-to-cov:c} of \cref{lem:bratteli-to-covering},
\[D_{n+1,0} \subseteq [p_{n,\min}] \subseteq D_{n,0}.\]
Thus, $\Dcal_{n+1,0}$ refines $\Dcal_{n,0}$.
We define a covering map $\fai_{n+1} : G_{n+1} \to G_n$ by
 $\fai(u) = v$ ($u \in V(G_{n+1}), v \in V(G_n)$) if and only if
 $u \subseteq v$.
To check the \bidirectionalitys condition,
 we only have to check at $D_{n+1,0}$.
The calculation
\[ \phi(D_{n+1,0}) \subseteq \phi([p_{n,\min}]) \in \Dcal_n\]
 shows the \pdirectionalitys at $D_{n+1,0}$.
The \pdirectionalitys and the calculation
\[ \phi^{-1}(D_{n+1,0}) = \bigcup_{v \in V_{n+1}} [p_{v,w(v)}]
 \subseteq [p_{n,\max}]\]
 show the \bidirectionalitys at $D_{n+1,0}$.
From the refining sequence
 $\Dcal_0 \prec \Dcal_1 \prec \Dcal_2 \prec \dotsb$,
 we have defined a KR-covering $\covrepa{G}{\fai}$, the rank of which
 is equal to the rank of the Bratteli diagram.
We also note \cref{item:bratteli-to-cov:d}.
Thus, by the argument in \cref{rem:partition-covering},
 we get the next lemma:
\begin{lem}\label{lem:BD-to-KR-cov}
Let $(V,E)$ be an essentially simple ordered Bratteli diagram.
Let $\covrepa{G}{\fai}$ be the \bidirectionals KR-covering defined above.
Then, $G_{\infty}$ is topologically conjugate to $(E_{0,\infty},\phi)$.
Moreover, it follows that the ranks coincide.
\end{lem}

\begin{thm}\label{thm:coincidence-of-ranks-essentially-simple-BV-and-KR-cov}
Let $(X,f)$ be a (homeomorphic) essentially minimal zero-dimensional system.
Then, the topological rank by essentially simple Bratteli--Vershik models and that by KR-covering models coincide.
\end{thm}
\begin{proof}
By \cref{lem:KR-covering->Bratteli-diagram}, we get 
\[ \text{topological rank (Bratteli--Vershik model)}
 \le \text{topological rank (KR-covering model)}.\]
The converse is obtained by \cref{lem:BD-to-KR-cov}.
\end{proof}

Owing to \cref{thm:essentially-simple-representation},
 the next proposition is no more than a remark:
\begin{prop}\label{prop:bidirectional-KR-cov-essentially-minimal}
 For a homeomorphic zero-dimensional topological dynamical system $(X,f)$,
 the following are equivalent:
\enumb
\item $(X,f)$ is essentially minimal,
\item $(X,f)$ has a Bratteli--Vershik representation of an essentially simple
 ordered Bratteli diagram, and
\item there exists a \bidirectionals KR-covering $\covrepa{G}{\fai}$ such that
 $G_{\infty}$ is topologically conjugate to $(X,f)$.
\enumn
\end{prop}
\begin{proof}
By \cref{thm:essentially-simple-representation,lem:KR-covering->Bratteli-diagram,lem:BD-to-KR-cov}, the proof is obvious.
\end{proof}
\section{Homeomorphic proximal Cantor systems}
\label{sec:proximal-systems}
We now focus our attention on proximal cases.
\begin{defn}\label{defn:proximal-BV}
We say that an ordered Bratteli diagram is \textit{proximal}
 if there exist vertices $v_{n,1} \in V_n$ $(n \bni)$ and edges
 $e_n \in E_n$ $(n \bpi)$ such that
\itemb
\item  $s(e_{n+1}) = v_{n,1}$ and $r(e_{n+1}) = v_{n+1,1}$ ($n \bni$),
\item the infinite path $p_0 := (e_{1}, e_{2},\dotsc)$ is both 
 maximal and minimal, and
\item $p_0$ is the unique maximal infinite path from $V_0$ and $p_0$ is 
 the unique minimal infinite path from $V_0$.
\itemn
It is evident that an essentially simple ordered Bratteli diagram
 is proximal if and only if $p_{\max} = p_{\min}$.
\end{defn}
\begin{defn}\label{defn:proximal-covering}
We say that a KR-covering $\covrepa{G}{\fai}$ is \textit{proximal} if
 for each $n \bpi$, where $\seb c_{n,t} \mid  1 \le t \le r_n \sen$ 
 is the set of circuits of $G_n$, there exists unique $c_{n,t}$
 such that $p(n,t) = 1$.
\end{defn}
\begin{nota}\label{nota:proximal-KR}
Hereafter, without loss of generality, we assume that
 $p(n,1) = 1$ for each $n \bpi$.
We rewrite $e_n := c_{n,1}$ and $c_{n,t-1} := c_{n,t}$ for $1 < t \le r_n$.
We write $\barr_n := r_n -1$.
Thus, we get $\barr_n+1$ circuits $e_n$ and $c_{n,t}$ ($1 \le t \le \barr_n$).
\end{nota}
We remark that the definition of
 \pdirectionalitys and \bidirectionalitys implies the following:
\begin{rem}\label{rem:proximal-KR}
Let $\covrepa{G}{\fai}$ be a proximal KR-covering.
It is evident that for each $n \bpi$, $\fai_{n+1}(e_{n+1}) = e_n$.
Let $n \bpi$ and $1 \le t \le \barr_{n+1}$.
For each $1 \le t \le \barr_{n+1}$, we have
 \[\fai_{n+1}(c_{n+1,t}) = c_{n,t,1} c_{n,t,2} \dotsb c_{n,t,\barw(n,t)}\]
 for some $\barw(n,t) > 1$
 and $c_{n,t,i} \in \sC_n$ ($1 \le i \le \barw(n,t)$).
Then, $c_{n,t,1} = e_n$.
If the covering is \bidirectional,
 then we also get $c_{n,t,\barw(n,t)} = e_n$.
\end{rem}
We shall show the following:
\begin{thm}\label{thm:proximal}
Let $(X,f)$ be a zero-dimensional homeomorphic topological dynamical system.
Then, the following are equivalent:
\enumb
\item\label{item:thm:proximal:td} $(X,f)$ is proximal,
\item\label{item:thm:proximal:BV}
 $(X,f)$ has a Bratteli--Vershik representation by
 a proximal ordered Bratteli diagram, and
\item\label{item:thm:proximal:KR}
 $(X,f)$ is topologically conjugate to the inverse limit of
 a \bidirectionals proximal KR-covering.
\enumn
\end{thm}
\begin{proof}
\textbf{\cref{item:thm:proximal:td} implies \cref{item:thm:proximal:BV}:}
Let $(X,f)$ be a zero-dimensional homeomorphic proximal system.
Because this is essentially minimal,
 by \cref{thm:essentially-simple-representation},
 $(X,f)$ has a Bratteli--Vershik representation by
 an essentially simple ordered Bratteli diagram.
Furthermore,
 by the last statement of \cref{thm:essentially-simple-representation},
 both $p_{\max}$ and $p_{\min}$ are points in the unique minimal set
 that is a fixed point.
Thus, $p_{\max} = p_{\min}$.
This concludes the proof.

\textbf{\cref{item:thm:proximal:BV} implies \cref{item:thm:proximal:td}:}
Let $(V,E)$ be a proximal ordered Bratteli diagram.
Then, it follows that $p_{\max} = p_{\min}$.
Because $(V,E)$ is essentially simple by definition,
 it follows that this fixed point is the unique minimal set.
From \cref{thm:proximal-system}, we get that the Bratteli--Vershik system
 is proximal.

\textbf{\cref{item:thm:proximal:BV} implies \cref{item:thm:proximal:KR}:}
Let $(V,E)$ be a proximal ordered Bratteli diagram.
Because we have already shown that \cref{item:thm:proximal:BV}
 implies \cref{item:thm:proximal:td},
 the Bratteli--Vershik system $(E_{0,\infty},\phi)$ is proximal.
On the other hand, because $(V,E)$
 is an essentially simple ordered Bratteli diagram,
 the argument of \cref{lem:BD-to-KR-cov} can be applied, and we get
 a KR-covering $\covrepa{G}{\fai}$ such that
 $G_{\infty}$ and $(E_{0,\infty},\phi)$ are topologically conjugate.
Thus, $G_{\infty}$ is proximal, i.e., it has a fixed point that is the
 unique minimal set.
It follows that each $G_n$ has a circuit of period 1.
This concludes the proof.

\textbf{\cref{item:thm:proximal:KR} implies \cref{item:thm:proximal:td}:}
Let $\covrepa{G}{\fai}$ be a \bidirectionals proximal KR-covering.
Let $v_{n,0}$ be the center of $G_n$ for all $n \bpi$.
Then, by the definition of proximality for KR-covering,
 $(v_{n,0},v_{n,0}) \in E(G_n)$ for all $n \bpi$.
This implies the existence of a fixed point.
By \cref{prop:bidirectional-KR-cov-essentially-minimal},
 this is the unique minimal set.
This completes the proof.
\end{proof}
\section{Array Systems}\label{sec:array-systems}
\label{sec:array-systems}
In this section,
 following \cite{DOWNAROWICZ_2008FiniteRankBratteliVershikDiagAreExpansive},
 we introduce the \textit{array system}.
Let $\covrepa{G}{\fai}$ be a \bidirectionals KR-covering.
For each $n \bni$, we have expressed
 the center of $G_n$ as $v_{n,0} \in V(G_n)$.
Further, we have denoted $G_{\infty} = (X,f)$.
For each $x \in X$, $n \ge 0$, and $i \bi$,
 there exists a unique $v \in V(G_n)$ such that $f^i(x) \in U(v)$.
We express this $v$ as $u_{n,i}$.
We define a sequence $\ddx := (u_{n,i})_{n \bni, i \bi}$.
We define $\ddX := \seb \ddx \mid x \in X\sen$
 and $\ddx(n,i) := u_{n,i}$ for $n \bni$ and $i \bi$.
Further, we denote $\ddx[n] := (u_{n,i})_{i \bi}$.
Because every $\ddx[n]$ is an infinite path of $G_n$,
 every segment $(\ddx[n])[a,b] = (u_{n,a},u_{n,a+1},\dotsc,u_{n,b})$
 is a finite path of $G_n$.
We define a shift map $\sigma : \ddX \to \ddX$ that shifts left.
Let $x \in X$, $n \ge 0$ and $i \bi$.
Then, there exists a unique circuit $c \in \sC_n = \sC(G_n)$ such that
 $(\ddx(n,i),\ddx(n,i+1)) \in E(c)$.
We denote $\barx(n,i) = c$.
Thus, once a $c \in \sC_n$ appears, then the $c$ continues at least $l(c)$ times.
If it is necessary to distinguish the beginning of the circuits,
 then it can be done by changing
 $\barx(n,i) = c$ into $\barx(n,i) = \check{c}$
 for all the $i$'s with $\ddx(n,i) = v_{n,0}$.
\begin{rem}\label{rem:proximal-cond-array}
In the case of proximal systems,
 we shall show, in \cref{rem:continuation-upper-bound}, that
 for each $n \bpi$,
 the number of continuations of the same $c \in \sC_n \setminus \seb e_n \sen$
 has an upper bound.
Therefore, in the case of proximal systems, we do not need this change.
\end{rem}
Following \cite{DOWNAROWICZ_2008FiniteRankBratteliVershikDiagAreExpansive},
 we make an $n$-\textit{cut} in each $\barx[n]$
 just before each $i$ with
 $\ddx(n,i) = v_{n,0}$ (see \cref{fig:array-system}).
The pair $(\barX,\sigma)$ of a set $\barX := \seb \barx \mid x \in X \sen$
 and the shift map $\sigma : \barX \to \barX$ that shifts left
 is called an \textit{array system}.
To abbreviate the notation, we write $e_0 := (v_{0,0},v_{0,0}) \in E(G_0)$.
Therefore, for each $x \in X$ and $n \bni$,
 there exists a unique sequence
 $\barx[n]  := (\dotsc,\barx(n,-2),\barx(n,-1),\barx(n,0),\barx(n,1),\dotsc)$
 of circuits of $G_n$ that is separated by the cuts.
For integers $s < t$,
 we denote $(\barx[n])[s,t] :=
 (\barx(n,s),\barx(n,s+1),\barx(n,s+2),\dotsc, \barx(n,t))$.
We have abbreviated $\barx[0] = (\dotsc,e_0,e_0,e_0,\dotsc)$
 that is cut everywhere.
%
%
\begin{figure}
\begin{center}\leavevmode
\xy
(-5,18)*{}; (105,18)*{} **@{-},
 (49,15)*{c_{n,1}
 \hspace{9mm} c_{n,3} \hspace{4mm}
 \hspace{12mm} c_{n,1} \hspace{6mm}
 \hspace{8mm} c_{n,3} \hspace{3mm}
 \hspace{10mm} c_{n,2} \hspace{4mm}
 \hspace{8mm} c_{n,1}
 \hspace{1mm} },
(7,18)*{}; (7,12)*{} **@{-},
(28,18)*{}; (28,12)*{} **@{-},
(49,18)*{}; (49,12)*{} **@{-},
(70,18)*{}; (70,12)*{} **@{-},
(84,18)*{}; (84,12)*{} **@{-},
(-5,12)*{}; (105,12)*{} **@{-},
 (54,9)*{c_{n+1,5}
 \hspace{35mm} c_{n+1,1} \hspace{3mm}
 \hspace{28mm} c_{n+1,3} \hspace{2mm} },
(28,12)*{}; (28,6)*{} **@{-},
(84,12)*{}; (84,6)*{} **@{-},
(-5,6)*{}; (105,6)*{} **@{-},
\endxy
\end{center}
\caption{$n$th and $(n+1)$th rows of an array system with cuts.}
\label{fig:array-system}
\end{figure}
%
%
%
%
%
For an interval $[n,m]$ with $m > n$, the combination of rows $\barx[n']$
 with $n \le n' \le m$ is denoted as $\barx[n,m]$.
The \textit{array system of}
 $x$ is the infinite combination $\barx[0,\infty)$
 of rows $\barx[n]$ for all $0 \le n < \infty$ (see \cref{fig:array-system-2}).
%
%
Note that for $m > n$, if there exists an $m$-cut at position $i$
 (just before position $i$), then
 there exists an $n$-cut at position $i$ (just before position $i$).
%
%
%
%
\begin{figure}
\begin{center}\leavevmode
\xy
(-5,30)*{}; (105,30)*{} **@{-},
 (49,27)*{e_0\phantom{{}_{{},f}}\ e_0\phantom{{}_{{},f}}
 \ e_0\phantom{{}_{{},l}}
 \ e_0\phantom{{}_{{},l}}\ e_0\phantom{{}_{{},l}}
 \ e_0\phantom{{}_{{},f}}\ e_0\phantom{{}_{{},l}}
 \ e_0\phantom{{}_{{},l}}\ e_0\phantom{{}_{{},l}}
 \ e_0\phantom{{}_{{},f}}\ e_0\phantom{{}_{{},l}}
 \ e_0\phantom{{}_{{},l}}\ e_0\phantom{{}_{{},l}}
 \ e_0\phantom{{}_{{},f}}\ e_0\phantom{{}_{{},l}}\ e_0},
(0,30)*{}; (0,24)*{} **@{-},
(7,30)*{}; (7,24)*{} **@{-},
(14,30)*{}; (14,24)*{} **@{-},
(21,30)*{}; (21,24)*{} **@{-},
(28,30)*{}; (28,24)*{} **@{-},
(35,30)*{}; (35,24)*{} **@{-},
(42,30)*{}; (42,24)*{} **@{-},
(49,30)*{}; (49,24)*{} **@{-},
(56,30)*{}; (56,24)*{} **@{-},
(63,30)*{}; (63,24)*{} **@{-},
(70,30)*{}; (70,24)*{} **@{-},
(77,30)*{}; (77,24)*{} **@{-},
(84,30)*{}; (84,24)*{} **@{-},
(91,30)*{}; (91,24)*{} **@{-},
(98,30)*{}; (98,24)*{} **@{-},
(105,30)*{}; (105,24)*{} **@{-},%
(-5,24)*{}; (105,24)*{} **@{-},
 (48,21)*{c_{1,3}
 \hspace{11mm} c_{1,3} \hspace{8mm}
 \hspace{8mm} c_{1,1} \hspace{10mm}
 \hspace{5mm} c_{1,3} \hspace{8mm}
 \hspace{4mm} c_{1,2} \hspace{8mm}
 \hspace{3mm} c_{1,1}
 \hspace{0mm} },
(7,24)*{}; (7,18)*{} **@{-},
(28,24)*{}; (28,18)*{} **@{-},
(49,24)*{}; (49,18)*{} **@{-},
(70,24)*{}; (70,18)*{} **@{-},
(84,24)*{}; (84,18)*{} **@{-},
(-5,18)*{}; (105,18)*{} **@{-},
(56,15)*{c_{2,3} \hspace{0.7mm}\
 \hspace{35mm} c_{2,1} \hspace{27mm}
 \hspace{11mm} c_{2,3}
 \hspace{2mm} },
(28,18)*{}; (28,12)*{} **@{-},
(84,18)*{}; (84,12)*{} **@{-},
(-5,12)*{}; (105,12)*{} **@{-},
(36,9)*{c_{3,3} \hspace{26mm}
 \hspace{15mm} c_{3,1}},
(28,12)*{}; (28,6)*{} **@{-},
(-5,6)*{}; (105,6)*{} **@{-},
(50,4)*{\vdots},
\endxy
\end{center}
\caption{First 4 rows of an array system.}\label{fig:array-system-2}
\end{figure}
%
%
%
%
%
\begin{nota}\label{nota:n-symbol}
For each circuit $c = c_{n,i} \in \sC_n$, if we write
$\fai_n(c) = c_{n-1,i,1} c_{n-1,i,2} \dotsb c_{n-1,i,w(n,i)}$ as 
 a series of circuits,
 then each $c_{n-1,i,j}$ determines a series of circuits of $G_{n-2}$
 similarly.
Thus, we can determine a set of circuits arranged in a square form as in
 \cref{fig:2-symbol}.
This form is said to be the \textit{$n$-symbol} and denoted by $c$.
For $m < n$,
 the projection $c[m]$ that is a finite sequence of circuits of $G_m$
 is also defined.
\end{nota}
\begin{figure}
\begin{center}\leavevmode
\xy
(28,38)*{}; (84,38)*{} **@{-},
(57,35)*{e_0\phantom{{}_{{},l}}\ e_0\phantom{{}_{{},l}}
 \ e_0\phantom{{}_{{},l}}\ e_0\phantom{{}_{{},l}}
 \ e_0\phantom{{}_{{},f}}\ e_0\phantom{{}_{{},l}}
 \ e_0\phantom{{}_{{},l}}\ e_0\phantom{{}_{{},l}}},
(28,38)*{}; (28,32)*{} **@{-},
(35,38)*{}; (35,32)*{} **@{-},
(42,38)*{}; (42,32)*{} **@{-},
(49,38)*{}; (49,32)*{} **@{-},
(56,38)*{}; (56,32)*{} **@{-},
(63,38)*{}; (63,32)*{} **@{-},
(70,38)*{}; (70,32)*{} **@{-},
(77,38)*{}; (77,32)*{} **@{-},
(84,38)*{}; (84,32)*{} **@{-},
(28,32)*{}; (84,32)*{} **@{-},
 (60,29)*{
 \hspace{7mm} c_{1,1} \hspace{9mm}
 \hspace{5mm} c_{1,3} \hspace{7mm}
 \hspace{3mm} c_{1,2} \hspace{8mm}
 \hspace{2mm} },
(28,32)*{}; (28,26)*{} **@{-},
(49,32)*{}; (49,26)*{} **@{-},
(70,32)*{}; (70,26)*{} **@{-},
(84,32)*{}; (84,26)*{} **@{-},
(28,26)*{}; (84,26)*{} **@{-},
(56,23)*{
 \hspace{23mm} c_{2,1} \hspace{23mm}},
(28,26)*{}; (28,20)*{} **@{-},
(84,26)*{}; (84,20)*{} **@{-},
(28,20)*{}; (84,20)*{} **@{-},
\endxy
\end{center}
\caption{$2$-symbol corresponding to the circuit $c_{2,1}$
 of \cref{fig:array-system-2}.}\label{fig:2-symbol}
\end{figure}
%
%
%
It is clear that $\barx[n] = \barx'[n]$
 implies that $\barx[0,n] = \barx'[0,n]$.
If $x \ne x'$ $(x, x' \in X)$,
 then there exists an $n > 0$ with $x[n] \ne x'[n]$.
For $x, x' \in X$,
 we say that the pair $(x,x')$ is \textit{$n$-compatible}
 if $x[n] = x'[n]$.
If $x[n] \ne x'[n]$,
 then we say that $x$ and $x'$ are \textit{$n$-separated}.
We recall that
 if there exists an $n$-cut at position $k$,
 then there exists an $m$-cut at position
 $k$ for all $0 \le m \le n$.
Let $x \ne x'$.
%
%
%
The set $\barX_n := \seb \barx[n] \mid x \in X \sen$ is a two-sided subshift
 of a finite set
 $\sC_n$.
The factoring map is denoted by $\pi_n : \barX \to \barX_n$,
 and the shift map is denoted by $\sigma_n : \barX_n \to \barX_n$.
We simply write $\sigma = \sigma_n$ for all $n$
 if there is no confusion.
For $m > n \ge 0$, the factoring map $\pi_{m,n} : \barX_m \to \barX_n$
 is defined by $\pi_{m,n}(\barx[m]) = \barx[n]$ for all $x \in X$.
\section{Proximal systems of rank 2.}\label{sec:prox-rank-2}
\label{sec:proximal-systems-of-rank-2}
\subsection{General finite-rank proximal systems}
First, we recall general notations for homeomorphic
 proximal zero-dimensional systems.
Let $\covrepa{G}{\fai}$ be a \bidirectionals proximal KR-covering and
 let $(X,f)$ be the inverse limit.
We exclude the trivial case in which $X$ consists of a single point.

In \cref{nota:proximal-KR},
 we have stated that for all $n \bpi$,
 $\sC_n = \seb e_n, c_{n,1}, \dotsc, c_{n,\barr_n} \sen$.
In \cref{rem:proximal-KR}, for all $n \bpi$ and $1 \le t \le \barr_{n+1}$,
 we have written that 
 \[\fai_{n+1}(c_{n+1,t}) = c_{n,t,1} c_{n,t,2} \dotsb c_{n,t,\barw(n,t)},\]
 for some $\barw(n,t) > 1$
 and $c_{n,t,i} \in \sC_n$ ($1 \le i \le \barw(n,t)$).
Further, by the \bidirectionalitys condition,
 it follows that $c_{n,t,1} = c_{n,t,\barw(n,t)} = e_n$.
In this section,
 each circuit $c_{n,i}$ $(1 \le i \le \barr_n)$ is represented as 
 $c_{n,i} =
 (u_{n,i,0} = v_{n,0},u_{n,i,1},u_{n,i,2},\dotsc,u_{n,i,\barp(n,i)} = v_{n,0})$
 $(i = 1,2,\dotsc,\barr_n)$.
Evidently,
 the vertices $u_{n,i,j}$ $(1 \le i \le \barr_n, 1 \le j \le \barp(n,i))$
 are mutually distinct,
 except $u_{n,i,0} = u_{n,i,\barp(n,i)} = v_{n,0}$,
 for all $i$ $(1 \le i \le \barr_n)$.
The unique fixed point is denoted as
 $p := (v_{0,0},v_{1,0},v_{2,0},\dotsc) \in X$.
We also write $c_{n,i}(j) = u_{n,i,j}$ for all $0 \le j \le \barp(n,i)$.
Thus, we get the following notation:
\begin{nota}
Let $c \in \sC_n$.
Then, for each $0 \le j \le l(c)$, we get a vertex $c(j) \in V(c)$.
\end{nota}
By telescoping, for $n = 1$, we can assume that $\barr_1 \ge 1$
 and $\barp(1,i) \ge 2$ $(1 \le i \le \barr_1)$.
Suppose that $G_n$ is defined.
Then, $\fai_{n+1}$ is defined
 in the following manner:
\itemb
\item $\fai_{n+1}(v_{n+1,0}) = v_{n,0}$ and
\item for each $1 \le k \le \barr_{n+1}$, there exists a positive integer
 $h(n,k)$ such that 
 $\fai_{n+1}(c_{n+1,k}) = c_1\ c_2\ \dotsb\ c_{h(n,k)}$
 for all $k\ (1 \le k \le r_{n+1})$,
 where $c_j \in \sC_n$ for all $1 \le j \le h(n,k)$ and
 $c_1 = c_{h(n,k)} = e_n$.
\itemn
\begin{rem}\label{rem:continuation-upper-bound}
Note that by the above argument, for each $n \bpi$, the number of
 continuations
 of the same $c \in \sC_n \setminus \seb e_n \sen$ such as $c c c \dotsb c$
has an upper bound.
\end{rem}
It follows that $h(n,k) \ge 2$ ($1 \le k \le \barr_{n+1}$).
Because the center is mapped to the center,
 it follows that $\fai_{n+1}(e_{n+1}) = e_n$.
For all $1 \le k \le \barr_{n+1}$,
 the period $\barp(n+1,k)$ $(1 \le k \le \barr_{n+1})$
 satisfies $\barp(n+1,k) = \sum_{1 \le j \le h(n,k)}l(c_j)$.
Let $N > 0$ be a positive integer.
For cycles $w_k$ $(1 \le k \le N)$ of $G_n$
 that start from and end at the central vertex $v_{n,0}$,
 we abbreviate $\prod_{k = 1}^N w_k := w_1 w_2 \dotsb w_N$.
The main examples we present in this paper are proximal systems of rank 2.
Nevertheless, in this subsection, we confirm some general properties
 of finite-rank proximal systems.
Let $\Gcal : \covrepa{G}{\fai}$ be a \bidirectionals KR-covering
  such that
 $(X,f)$ is (topologically conjugate to) the inverse limit.
Let $r > 1$ be the topological rank of $(X,f)$.
Then, we can take $\Gcal$ such that the rank of $\Gcal$ is $r$.
By telescoping, we can assume that $\num{\sC_n} = r$ for all $n \bpi$.
By \cref{thm:finite-ergodicity},
 there exist at most $r$ ergodic measures in $(X,f)$.
For finite-rank Bratteli--Vershik systems,
 in \cite{Shimomura_2017FiniteRankBVwithPeriodicityExpansive},
 we have shown the following:
\begin{thm}[\cite{Shimomura_2017FiniteRankBVwithPeriodicityExpansive}]
\label{thm:finite-expansive}
Let $(X,f)$ be a finite-rank homeomorphic zero-dimensional system
 such that no minimal set is an infinite odometer.
Then, $(X,f)$ is expansive.
\end{thm}
As a corollary, we have obtained the following:
\begin{thm}[\cite{Shimomura_2017FiniteRankBVwithPeriodicityExpansive}]
\label{thm:finite-expansive-proximal}
A finite-rank homeomorphic zero-dimensional proximal system is expansive.
\end{thm}
We show that finite-rank homeomorphic proximal systems have topological entropy
 $0$.
\begin{prop}\label{prop:topological-entropy-0}
Let $(X,f)$ be a finite-rank homeomorphic proximal system.
Then, it follows that the topological entropy $h(f) = 0$.
\end{prop}
\begin{proof}
Let $r > 1$ be the rank of $\Gcal$.
For $n \bpi$ and $l > 0$,
 let 
\[N_{n,l} := \num{\seb w \mid w \text{ is a walk in } G_n \text{ with }
 l(w) = l \sen}.\]
Then, $h(f) \le \limsup_{n \to \infty}
 \limsup_{l \to \infty}(1/l)\log(N_{n,l})$.
Because $l(c_{n,i}) \to \infty$ ($1 \le i \le r$) uniformly as $n \to \infty$,
 we get $\limsup_{n \to \infty}\limsup_{l \to \infty}(1/l)\log(N_{n,l}) = 0$.
\end{proof}

\subsection{Rank 2 implies residually scrambled.}
\label{subsec:representation}
Let $\Gcal : \covrepa{G}{\fai}$
 be a \bidirectionals proximal KR-covering of rank 2.
We recall that $p \in X$ is the unique fixed point.
By telescoping, we can write $\sC_n = \seb e_n, c_n \sen$
 for all $n \bpi$.
Further, for all $n \bpi$, we can write 
\[
 \fai_{n+1}(c_{n+1}) =
 e_n^{a(n,0)}\ \prod_{j = 1}^{b(n)} \left(c_n e_n^{a(n,j)}\right)
\]
 where $a(n,0) > 0$, $a(n,j) \ge 0 \text{ for each } 0 < j < b(n)$,
 $a(n,b(n)) > 0$, and $b(n) \ge 1$.
The next lemma is obvious:
\begin{lem}\label{lem:final-series}
For all $m > n$, it follows that
 $\fai_{m,n}(c_m) = e_n^{m-n}\ \dotsb\ e_n^{m-n}$.
\end{lem}
\begin{proof}
We omit the proof.
\end{proof}
By \cref{thm:finite-expansive-proximal}, it is clear that
 $(X,f)$ is expansive.
In the case of topological rank 2, we show that the level of expansiveness
 is from $X[1]$ as follows:
\begin{prop}\label{prop:expansive-from-1}
Let $\Gcal : \covrepa{G}{\fai}$
 be a \bidirectionals proximal KR-covering of rank 2.
We assume that every $G_n$ ($n \bpi$) has rank 2.
Let $(X,f)$ be the inverse limit.
Then, $(X,f)$ is topologically conjugate to $(\barX_1,\sigma)$.
\end{prop}
\begin{proof}
It follows that $(X,f)$ is topologically conjugate to the array system
 $(\barX,\sigma)$.
Let $x, y \in X$ be distinct elements.
Then, there exists an $n \ge 1$ such that $\barx[n] \ne \bary[n]$.
First, exchange any continuous occurrence of $e_n$ of length $\ge l(c_n)$
 by continuous occurrence of a symbol $a$ of the same length.
Other positions should be changed to the symbol $b$.
This transformation is denoted as $\alpha : \barX_n \to \seb a,b \sen^{\Z}$.
By \cref{lem:final-series},
 for each $\barz[n] \in \barX_n$, there exists an arbitrarily
 large continuation of $a$ in $\alpha(\barz[n])$.
Further, there exists no infinite continuation of the symbol $b$ in 
 $\alpha(\barz[n])$.
Suppose that $\alpha(\barx[n]) \ne \alpha(\bary[n])$.
Then, some continuation of $a$ of one of $\barx[n],\bary[n]$ must encounter
 an occurrence of $b$ of the other at the same position.
Without loss of generality, $(\barx[n])[s,t]$ is a continuation of $e_n$ of
 length $\ge l(c_n)$, and there exists $s \le u \le t$ such that
 $(\bary[n])(u) = c_n$.
If $(\bary[n])[s,t]$ is a continuation of $c_n$, then $(\bary[n])[s,t]$
 has to contain the whole of $c_n$.
Thus, $(\pi_{n,1}(\bary[n]))[s,t] = (\bary[1])[s,t]$ has to contain $c_1$.
Thus, we get $\barx[1] \ne \bary[1]$.
If $(\bary[n])[s,t]$ is not a continuation of $c_n$,
 then a continuation of $c_n$ ends at some $u'$ with $s \le u' < t$
 and $(\bary[n])(u'+1) = e_n$,
 or a continuation of $c_n$ begins from some $u'$ with $s < u' \le t$
 and $(\bary[n])(u'-1) = e_n$.
In both cases, we get $\barx[1] \ne \bary[1]$.
Thus, we only need to consider the case in which
 $\alpha(\barx[n]) = \alpha(\bary[n])$.
Because $\barx[n] \ne \bary[n]$, there exists an $i_0 \bi$ such that
 $\alpha(\barx[n])(i_0) = \alpha(\bary[n])(i_0) = b$ and
 $(\barx[n])(i_0) \ne (\bary[n])(i_0)$.
Take the maximal interval $i_0 \in [s,t]$ such that
 $\alpha(\barx[n])(i) = b$ for all $i \in [s,t]$.
Then, $\alpha(\bary[n])(i) = b$ for all $i \in [s,t]$.
It follows that in both $\barx[n]$ and $\bary[n]$, $s$ is the initial 
 position of the whole circuit $c_n$.
Nevertheless,
 there exists the least $j \ge s$ such that
 $(\barx[n])(j) \ne (\bary[n])(j)$.
Now, it is easy to see that $(\barx[1])[s,t] \ne (\bary[1])[s,t]$.
\end{proof}
Suppose that $b(n) > 1$ only for finitely many times.
Then, by telescoping, we have a representation by \bidirectionals KR-covering
 with $b(n) = 1$ for all $n \bpi$.
In this case, it is easy to see that the inverse limit $(X,f)$ consists of
 only one orbit except the unique fixed point $p$.
In this paper, we avoid this case:
\begin{nota}\label{nota:Cantor}
Hereafter,
 we consider rank 2 proximal systems that are Cantor systems.
\end{nota}
By \cref{thm:finite-ergodicity},
 there exist at most $2$ ergodic measures in $(X,f)$.
%
%
By \cref{prop:expansive-from-1}, the inverse limit $(X,f)$ of $\Gcal$
 is isomorphic to $(\barX_1,\sigma)$.
Thus, $(X,f)$ is isomorphic to a subshift of two symbols
 $\seb e_1, c_1 \sen$.
By \cref{prop:topological-entropy-0}, they have topological entropy $0$.
%
%
%

We would like to show that the substitution subshift that is mentioned in \cite[Proposition 55]{Blanchard_2008TopSizeOfScrambSet} has topological rank 2.
For $l \ge 1$, we define
 $W^l :=
 \seb a_1 a_2 \dotsb a_l \mid
 a_i \in \seb 0,1 \sen \text{ for all } 1 \le i \le l \sen$.
Further, we define $W^* := \bigcup_{l \ge 1}W^l$.
The substitution dynamical system given
 in \cite[Proposition 55]{Blanchard_2008TopSizeOfScrambSet}
 is defined by the non-primitive substitution
 $\tau : \seb0, 1\sen  \to W^∗$ as
 $\tau (0) = 001$ and $\tau(1) = 1$.
If we define $\alpha : \seb 0,1 \sen \to W^*$
 by setting $\alpha := \tau^2$, then we get
 $\alpha(1) = 1$ and $\alpha(0) = 0010011$.
By $\Lambda_{\alpha}$, we denote the set of all $x \in \seb 0,1 \sen^{\Z}$
 such that every finite sub-block of $x$ is contained in $\alpha^k(0)$
 for some $k > 0$.
Then, the subshift $(\Lambda_{\alpha},\sigma)$ coincides
 with the substitution
 subshift defined in \cite[Proposition 55]{Blanchard_2008TopSizeOfScrambSet}.
We show the following:
\begin{prop}\label{prop:proposition55-is-rank-2}
The substitution dynamical system $(\Lambda_{\alpha},\sigma)$ has 
 topological rank $2$.
\end{prop}
\begin{proof}
We define a substitution as $\beta(1) = 1$ and $\beta(0) = 1001001$.
If we use $\beta$ in place of $\alpha$, we get the set $\Lambda_{\beta}$.
First, we show that $\Lambda_{\alpha} = \Lambda_{\beta}$.
As usual, we can extend the map $\alpha$ to the map
 $\alpha : W^* \to W^*$ such that for all $l \ge 1$,
 $\alpha(a_1 a_2 \dotsb a_l) = \alpha(a_1) \alpha(a_2) \dotsb \alpha(a_l)$.
We also extend the map $\beta$ to the map $\beta : W^* \to W^*$.
It is evident that $\alpha^k(1)=1$ and $\beta^k(1)=1$ for all $k > 0$.
It is easy to check the calculation
 $\beta(\alpha(0))=\alpha(\beta(0))$.
Thus, it is evident that $\beta \circ \alpha = \alpha \circ \beta$.
For $k \ge 1$, we denote $1^k := \underbrace{11\dotsb 1}_k$.
It is easy to see that for all $l \ge 1$, there exists $k > 0$ such that
 the block $1^l 0$ appears in $\alpha^k(0)$.
For $l > k \ge 1$, we show that $\alpha^k(1^l 0) = \beta^k(1^{l-k}0 1^k)$.
For $k=1$, we calculate
\[\alpha(1^l 0) = 1^l \alpha(0) = 1^1\ 0010011
 = 1^{l-1}\ 1001001\ 1 = 1^{l-1}\ \beta(0)\ 1.\]
Thus, the claim is satisfied.
Assume that the claim is satisfied for some $k \ge 1$.
For $l > k+1$, we calculate
\begin{equation}
\begin{split}
\alpha^{k+1}(1^l 0) & = \alpha^k(1^{l-1} \alpha(10))
  = 1^{l-1} \alpha^k(10010011)\\
 & = 1^{l-1} \alpha^k(\beta(0)) 1  = 1^{l-1} \beta(\alpha^k(0)) 1 = \beta(\alpha^k(1^{l-1} 0)) 1\\
 &  = \beta(\beta^k(1^{l-1-k} 0 1^k) 1
 = \beta^{k+1}(1^{l-k-1} 0 1^{k+1}).
\end{split}
\end{equation}
This shows that the claim is satisfied for $k+1$.
Now, it is easy to see that $\Lambda_{\alpha} = \Lambda_{\beta}$.
Hereafter, we construct a graph covering such that
 the inverse limit is topologically conjugate to $\Lambda_{\beta}$.
Let $l(c_1) = 2$,
 $\fai_{2}(c_{2}) = e_1 c_1 e_1 c_1 e_1$, and
 $\fai_{n+1}(c_{n+1}) = e_n c_n^2 e_n c_n^2 e_n$ for $n \ge 2$.
Let $(X,f)$ be the inverse limit of this graph covering.
By exchanging $e_1$ with $1$ and $c_1$ with $0$ in $\barX_1$, we get a subshift
 $\Lambda$.
It is not difficult to see that $\Lambda = \Lambda_{\beta}$.
\end{proof}
\begin{nota}
We denote 
 $W^s(p) := \seb x \in X \mid \lim_{i \to +\infty}f^i(x) = p \sen$,
 and 
 $W^u(p) := \seb x \in X \mid \lim_{i \to +\infty}f^{-i}(x) = p \sen$.
\end{nota}
\begin{lem}\label{lem:positively-expansive}
If $x,y \in X \setminus W^s(p)$ with $x \ne y$,
 then $\limsup_{i \to +\infty}d(f^i(x),f^i(y)) > 0$.
If $x,y \in X \setminus W^u(p)$ with $x \ne y$,
 then $\limsup_{i \to +\infty}d(f^{-i}(x),f^{-i}(y)) > 0$.
\end{lem}
\begin{proof}
We show the first statement. The last statement follows in the same manner.
Take $x,y \in X \setminus W^s(p)$ such that $x \ne y$.
We show that $\limsup_{i \to +\infty}d(f^i(x),f^i(y)) > 0$.
Suppose that, in contrast, $\limsup_{i \to +\infty}d(f^i(x),f^i(y)) = 0$.
Then, there exists the least $i_0 \bi$ such that $(\barx[1])(i) = (\bary[1])(i)$
 for all $i \ge i_0$.
It follows that $(\barx[1])(i_0-1) \ne (\bary[1])(i_0 -1)$.
Without loss of generality, we assume that 
 $(\barx[1])(i_0-1) = c_1$ and $(\bary[1])(i_0-1) = e_1$.
Take an $n > 1$ arbitrarily.
Then, the least $n$-cut of $\barx[n]$ in the region $[i_0, \infty)$
 is strictly larger than $i_0$.
Obviously, $(\barx[n])(i) = (\bary[n])(i)$ for all sufficiently large $i$.
Therefore, there exists the least $i_1 > i_0$ such that 
 for all $i \ge i_1$, the $n$-cuts are the same.
Obviously, $i_1$ itself is the position of a common $n$-cut of $\barx[n]$
 and $\bary[n]$.
It follows that $(\barx[n])(i_1 -1) \ne (\bary[n])(i_1-1)$.
We get the next two cases:
\enumb
\item\label{i:xcn} $(\barx[n])(i_1 -1) = c_n$ and $(\bary[n])(i_1 -1) = e_n$, and
\item\label{i:ycn} $(\barx[n])(i_1 -1) = e_n$ and $(\bary[n])(i_1 -1) = c_n$.
\enumn
Suppose that \cref{i:xcn} holds.
Take the largest $i_2 \le i_1$ with $(\barx[1])(i_2) = c_1$.
We consider the next two cases:
\enumb
\setcounter{enumi}{2}
\item\label{i:xcn-final} $i_2 = i_0-1$, and
\item\label{i:xcn-not-final} $i_2 \ge i_0$.
\enumn
Suppose that \cref{i:xcn-final} holds.
Then, $(\barx[1])(i) = e_1$ for $i_0 \le i \le i_1$.
Because $i_1$ is the position of an $n$-cut,
 by \cref{lem:final-series}, we get $i_1 - i_0 \ge n-1$.
Thus, we get $(\barx[1])(i) = e_1$ for $i_0 \le i < i_0 + n -1$.
Suppose that \cref{i:xcn-not-final} holds.
Because the largest $i < i_1$ with $(\bary[n])(i) = c_n$
 is strictly less than $i_1-1$,
 the largest $i < i_1$ with $(\bary[1])(i) = c_1$
 is strictly less than $i_2$.
This implies that $(\barx[1])(i_2) \ne (\bary[1])(i_2)$ and $i_2 \ge i_0$,
 which is a contradiction.
Thus, in the case of \cref{i:xcn}, we get $(\barx[1])(i) = e_1$ for
 $i_0 \le i < i_0 + n -1$.
Next, we suppose that \cref{i:ycn} holds.
Because $(\barx[1])(i_0-1) = c_1$ and $(\barx[n])(i_1 -1) = e_n$,
 by \cref{lem:final-series}, we get $i_1 - 1 - (i_0 -1) > n -1$,
 i.e., $i_1 - i_0 \ge n$.
Take the largest $i_3 < i_1-1$ with $(\bary[1])(i_3) = c_1$.
Then, because $i_1$ is a position of a common $n$-cut,
 we get $(\barx[1])(i_3) = e_1$.
Because $(\barx[1])(i) = (\bary[1])(i)$ for all $i \ge i_0$,
 it follows that $i_3 < i_0$.
In particular, $(\bary[1])(i) = e_1$ for all $i_0 \le i \le i_1 -1$.
Because $(\barx[1])(i) = (\bary[1])(i)$ for all $i \ge i_0$,
 we get $(\barx[1])(i) = e_1$ for all $i_0 \le i \le i_1 -1$.
Thus, we get $(\barx[1])(i) = e_1$ for all $i_0 \le i \le i_0 + n -1$.
Thus, in both cases of \cref{i:xcn,i:ycn},
 we get $(\barx[1])(i) = e_1$ for all $i_0 \le i < i_0 + n -1$.
Because $n > 1$ is arbitrary, we can conclude that
 $\lim_{i \to +\infty} f^i(x) = p$, which is a contradiction.
\end{proof}
\begin{lem}\label{lem:fatting}
Let $x \in X$ and $a < b$ be integers.
Then, there exists an $N > 0$ such that for all $n \ge N$,
 the sequence
 $(\barx[1])[a,b]$
 appears in $c_n[1]$ for the $n$-symbol $c_n$.
\end{lem}
\begin{proof}
The ranges of the $n$-cuts of $c_n$ are extended, by at least one,
 to the left and to the right,
 as $n$ increases.
Thus, the proof is obvious.
\end{proof}
We recall that $(X,f)$ is a Cantor system.
\begin{lem}\label{lem:single-orbit}
Both $W^s(p) \setminus \seb p \sen$ and $W^u(p) \setminus \seb p \sen$
 consist of single orbits.
\end{lem}
\begin{proof}
We show that $W^s(p) \setminus \seb p \sen$ consists of a single orbit.
Let $x \in W^s(p)$ with $x \ne p$.
Because $X$ is the inverse limit of $\covrepa{G}{\fai}$,
 $x = (v_0,v_1,v_2, \dotsc)$ with $v_n \in G_n$ for all $n \bni$.
Without loss of generality, we assume that
 $(\barx[1])(0) = c_1$ and $(\barx[1])(i) = e_1$ for all $i \ge 1$.
We show that such an $x$ is unique.
Let $n > 1$.
It is obvious that $(\barx[n])(0) = c_n$.
Furthermore, there exists an $i_n > 0$ such that
 $(\barx[n])(i) = c_n$ for all $0 \le i \le i_n$ and 
 $(\barx[n])(i) = e_n$ for all $i > i_n$.
Thus, there exists a unique $j$ ($0 < j < l(c_n)$) such that
 $v_n = c_n(j)$.
Because $n > 1$ is arbitrary, we get a unique $x \in X$.
\end{proof}
\begin{lem}\label{lem:distinct-orbits}
It follows that $W^s(p) \cap W^u(p) = \seb p \sen$.
\end{lem}
\begin{proof}
Let $x \in W^s(p) \cap W^u(p)$ such that $x \ne p$.
Because $x \ne p$, there exist integers $a < b$ such that
 the sequence
 $(\barx[1])[a,b] :=
 ((\barx[1])(a), (\barx[1])(a+1), (\barx[1])(a+2), \dotsc, (\barx[1])(b))$,
 covers all the appearances of $c_1$ in all the $\barx[1]$.
By \cref{lem:fatting}, there exists an $n > 1$ such that
 $(\barx[1])[a,b]$ is a part of $c_n[1]$.
Nevertheless, for some $m > n$, there exist at least two appearances
 of $c_n$ in $\fai_{m,n}(c_m)$.
This implies that there exists an extra $c_1$ in $\barx[1]$,
 which is a contradiction.
\end{proof}
\begin{thm}\label{thm:residually-scrambled}
Let $(X,f)$ be a (homeomorphic) proximal Cantor system with topological
 rank 2.
Let $K : = X \setminus (W^s(p) \cup W^u(p))$.
Then, $K$ is scrambled in both ways.
In particular, $(X,f)$ is residually scrambled.
Furthermore, $(X,f)$ is expansive and isomorphic to $(\barX_1,\sigma)$,
 which is a subshift of two-symbols.
\end{thm}
\begin{proof}
Let $K^s := X \setminus W^s(p)$ and $x,y \in K^s$ with $x \ne y$.
Because $X$ is homeomorphic to the Cantor set, by \cref{lem:single-orbit},
 $K^s$ is a dense $G_{\delta}$ subset.
We show that $K^s$ is positively scrambled.
Because $(X,f)$ is proximal, we get 
 $\liminf_{i \to +\infty}d(f^i(x),f^i(y)) = 0$.
By \cref{lem:positively-expansive}, we get
 $\limsup_{i \to +\infty}d(f^i(x),f^i(y)) > 0$.
Similarly, $K^u := X \setminus W^u(p)$ is dense $G_{\delta}$ and
 negatively scrambled.
Thus, $K = K^s \cap K^u$ is dense $G_{\delta}$ and scrambled in both ways.
By \cref{prop:expansive-from-1},
 $(X,f)$ is expansive and isomorphic to $(\barX_1,\sigma)$, which
 is a subshift of two symbols.
\end{proof}
\begin{rem}\label{rem:transitive}
Every orbit, except the fixed point $p$, of $(X,f)$ is dense.
To see this, let $x \notin W^s(p)$.
Then, it is obvious that for all $n \bpi$, $(\barx[n])[0,\infty)$ contains
 infinite $c_n$.
Thus, $x$ is positively transitive.
Next, let $x \in W^s(p)$ with $x \ne p$.
Then, by \cref{lem:distinct-orbits}, it follows that $x \notin W^u(p)$.
Thus, the same argument in the negative direction
 shows that the orbit of $x$ is dense.
\end{rem}
By \cref{thm:transitive-fixed},
 every transitive proximal system is densely uniformly chaotic.
By the above remark,
 every proximal Cantor system with topological rank 2 is densely uniformly
 chaotic.
\begin{rem}\label{rem:Li-Yorke-sensitive-2}
By \cref{rem:Li-Yorke-sensitive},
 a homeomorphic transitive proximal Cantor system $(X,f)$
 with finite topological rank is Li--Yorke sensitive.
In particular, Cantor homeomorphic proximal systems with topological rank 2
 are Li--Yorke sensitive.
\end{rem}
\begin{rem}
Let $S$ be a scrambled set of $(X,f)$.
Obviously, we get $\num{S \cap W^s(p)} \le 1$.
Because $W^s(p)$ is dense, $S$ does not have a non-empty interior.
\end{rem}
\subsection{Some examples}
In this subsection, we consider examples of proximal Cantor systems
 of rank 2.
We would like to manage the topological (weakly) mixing property together 
 with the number of ergodic measures.
We failed to construct an example that is not topologically mixing, yet
 weakly mixing, and has two ergodic measures.
We restrict our attention to the systems in which sequences $s(n), s'(n), t(n), \text{ and } t'(n)$ of positive integers exist, with
 both $t(n) \ge 2$ and $t'(n) \ge 2$,
 and the representing proximal KR-covering
 $\Gcal : \covrepa{G}{\fai}$ satisfies the following requirements:
\itemb
\item for all $n \ge 1$,
 $\fai_{n+1}(c_{n+1})
 = e_{n}^{s(n)} c_n^{t(n)} a_n c_n^{t'(n)} e_n^{s'(n)}$.
\itemn
Here, $a_n$ is a cycle that consists of $e_n$'s and $c_n$'s.
The number of $e_n$ that appears in $a_n$ is denoted as $s''(n)$, and
 the number of $c_n$ that appears in $a_n$ is denoted as $t''(n)$.
Because every vertex is covered at least twice,
 the resulting zero-dimensional system is a Cantor system.
Because we are considering systems of rank 2,
 we denote $l_n := l(c_n)$ ($n \bpi$) and
 $c_n = (v_{n,0} := v_{n,0}, v_{n,1}, v_{n,2}, \dotsc, v_{n,l_n} := v_{n,0})$.
\begin{nota}
The class of zero-dimensional systems with topological rank 2 having proximal
 KR-covering $\Gcal$ as above is denoted by $\Scal$.
\end{nota}
Because each $\fai_n$ is \bidirectional, $f$ is a homeomorphism.
Let $(X,f) \in \Scal$ and $\Gcal : \covrepa{G}{\fai}$ be its graph covering.
Let $x \in X$ be represented as $x = (v_0,v_1,v_2,\dotsc) \in X$.
For each $n \bpi$, we write $v_n = v_{n,j(n)}$ for some $j(n) \in [0,l_n)$.
There exists a point measure $\mu_p$ on the fixed point.
It is easy to see that $\mu_p = \lim_{n \to \infty}\tile_n$.
Suppose that there exist exactly two ergodic measures.
Let $\mu_c$ be the ergodic measure with $\mu_c \ne \mu_p$.
Then, by \cref{thm:ergodic-measures-are-from-essential-circuits},
 there exists a sequence $(\tilc'_n)_{n \bpi}$
 with $c'_n = e_n$ or $c'_n = c_n$ such that $\lim_n(c'_n) = \mu_c$.
If there is an infinite number of $n \bpi$ such that
 $c'_n = e_n$, then we get $\mu_c = \mu_p$, which is a contradiction.
Therefore, we get $\lim_{n \to \infty} \tilc_n = \mu_c$.
For $m > n$, without the assumption that $\mu_c \ne \mu_p$,
 we compute $\xi_{m,n}(\tilc_m)$.
Let us denote $\bars(n) := s(n) + s'(n) + s''(n)$ and
 $\bart(n) := t(n) +t'(n) + t''(n)$.
Further, let us denote $r(n) := (l_{n+1} - \bars(n))/l_{n+1}
 = (\bart(n)l_n)/(\bars(n)+ \bart(n)l_n)$.
Then, it follows that $1 - r(n) = \bars(n)/(\bars(n) + \bart(n)l_n)$.
By using the notation in \cref{subsec:invariant-measures},
 we can calculate 
 $\xi_{n+1,n}(\tilc_{n+1}) =  (1-r(n)) \tile_n + r(n) \tilc_n$.
Thus, we get
\begin{equation}\label{eqn:ratio}
 \xi_{m,n}(\tilc_m) = (1-r(m,n)) \tile_n + r(m,n) \tilc_n,
 \text{ where } r(m,n) := \prod_{i = n}^{m-1}r(i).
\end{equation}
Therefore, $\lim_{n \to \infty}\tilc_n = \mu_p$ is equivalent to
 $r(m,n) \to 0$ as $m \to \infty$ for all $n \bpi$.
By elementary analysis, this is equivalent to
\begin{equation}\label{eqn:ergodic-basic}
\sum_{i = 1}^{\infty}(1-r(i)) = \infty.
\end{equation}
The argument above still holds for arbitrary proximal Cantor
 systems with topological rank 2.
For general proximal Cantor systems with topological rank 2,
 we denote 
\[
 \fai_{n+1}(c_{n+1}) =
 e_n^{a(n,0)}\ \prod_{j = 1}^{b(n)} \left(c_n e_n^{a(n,j)}\right)
 \text{, for all } n \bpi,
\]
 where $a(n,0) > 0$, $a(n,j) \ge 0 \text{ for each } 0 < j < b(n)$,
 $a(n,b(n)) > 0$, and $b(n) \ge 1$.
In this case, we denote
\[r(n) := \frac{\left(l_{n+1} - \sum_{j = 0}^{b(n)} a(n,j)\right)}{l_{n+1}},\]
for all $n \bpi$.
Thus, we get the following:
\begin{thm}\label{thm:uniquely-ergodic}
A proximal Cantor system $(X,f)$ with topological rank 2
 is uniquely ergodic if and only if
 $\sum_{i = 1}^{\infty}(1-r(i)) = \infty$.
\end{thm}
Now, we return to the restricted systems $\Scal$.
We denote the subclass of $\Scal$ that consists of uniquely ergodic systems
 as $\Scalue$ and denote $\Scalna = \Scal \setminus \Scalue$.
The fact that every system in $\Scalna$ has a non-atomic measure is shown in 
 \cref{lem:K-has-measure-one}.
\begin{lem}\label{lem:exist-uniquely-ergodic-measures}
Both $\Scalue$ and $\Scalna$ are not empty.
\end{lem}
\begin{proof}
The proof is obvious from \cref{thm:uniquely-ergodic}:
 if we want to construct the systems in $\Scalue$,
 we only need to set $\bars(n)$ $(n \ge 1)$ to be very large compared
 with $\bart(n)$.
For the same reason,
 it is obvious that $\Scalna$ is not empty.
\end{proof}
\begin{lem}\label{lem:K-has-measure-one}
Suppose that $\mu_c \ne \mu_p$.
It follows that $\mu_c(K) = 1$ and $\mu_c(U) > 0$ for all opene $U$.
\end{lem}
\begin{proof}
Because $\tilc_n$ attaches a weight $1/l_n$ to $v_{n,0}$, it follows that $\mu_c(\seb p \sen) = 0$.
Because $W^s(p)$ consists of $\seb p \sen$ and another orbit,
 we get $\mu_c(W^s(p))= 0$.
Similarly, we conclude that $\mu_c(W^u(p))=0$.
Therefore, we get $\mu_c(K) = 1$.
Let $U$ be opene.
Then, there exist $n$ and $v \in V(G_n)$ such that $U(v) \subset U$.
As in \cref{nota:measure-inverse-limit-of-circuits},
 we can obtain an expression $\bar{\mu_c} = (\mu_{c,0},\mu_{c,1},\mu_{c,2},\dotsc)$,
 where $\mu_{c,n} = a_n \tile_n + b_n \tilc_n$
 for some $a_n, b_n \ge 0$ with $a_n +b_n =1$.
Because $\mu_c \ne \mu_p$, for all $n \bpi$, we get $b_n > 0$.
Therefore, we get $\mu_c(U(v)) > 0$ for all $v \in V(c_n)$ and $n \bpi$,
 which concludes the proof.
\end{proof}
Next, we analyze the mixing properties.
It is easy to see that the number of occurrences of $c_n$ in $\fai_{m,n}(c_m)$
 tends to infinity as $m \to \infty$.
\begin{nota}
For $m > n$, we denote $c_{m,n} := \fai_{m,n}(c_m)$ and $c_{n,n} := c_n$.
\end{nota}
We calculate the following:
\begin{equation}\label{eqn:descend2-original}
\begin{split}
 c_{n+2,n} & =
  \fai_{n+1}\left(e_{n+1}^{s(n+1)}\ c_{n+1}^{t(n+1)}\ a_{n+1}\ 
  c_{n+1}^{t'(n+1)}\ e_{n+1}^{s'(n+1)}\right)\\
  & = e_n^{s(n+1)}\ \left(\fai_{n+1}(c_{n+1})\right)^{t(n+1)}\ 
  \fai_{n+1}(a_{n+1})\  \left(\fai_{n+1}(c_{n+1})\right)^{t'(n+1)}\ 
  e_n^{s'(n+1)}\\
  & = e_n^{s(n+1)}\
  \left(e_n^{s(n)}\ c_n^{t(n)}\ a_n\ c_n^{t'(n)}\ e_n^{s'(n)}\right)^{t(n+1)}\\
  & \qquad \qquad \fai_{n+1}(a_{n+1})\ \left(e_n^{s(n)}\ c_n^{t(n)}
  \ a_n\ c_n^{t'(n)}\ e_n^{s'(n)}\right)^{t'(n+1)}\ 
  \ e_n^{s'(n+1)}.
\end{split}
\end{equation}
Therefore, we get
\begin{multline}\label{eqn:descend2}
 c_{n+2,n}
 = e_n^{s(n+1) + s(n)}\ 
 \underset{
   c_n^{t(n)} \text{ appears at first; and } c_n^{t'(n)}, \text{ in the end}
 }{
 \underbrace
 {
 c_n^{t(n)}\ a_n\ c_n^{t'(n)}\ e_n^{s(n)+s'(n)}\ c_n^{t(n)}\ a_n\ \dotsb c_n^{t'(n)}
 }
}
 \ e_n^{s'(n+1)+s'(n)}.
\end{multline}
Let us denote
\begin{equation}\label{eqn:descend2-by-d}
 c_{n+2,n} = e_n^{s(n+1)+s(n)}\ d_{n+2,n}\ e_n^{s'(n+1)+s'(n)}.
\end{equation}
Then, by \cref{eqn:descend2-original}, we get
\begin{equation}\label{eqn:descend2-by-d2}
\begin{split}
d_{n+2,n}
 & := \left(c_n^{t(n)}\ a_n\ c_n^{t'(n)}\ e_n^{s(n)+s'(n)}\right)^{t(n+1)-1}
 \ c_n^{t(n)}\ a_n\ c_n^{t'(n)}\\
 & \qquad \quad \quad e_n^{s'(n)}\ \fai_{n+1}(a_{n+1})\ e_n^{s(n)}\\ 
 & \qquad \quad \quad \quad 
 \left(c_n^{t(n)}\ a_n\ c_n^{t'(n)}\ e_n^{s(n)+s'(n)}\right)^{t'(n+1)-1}
 c_n^{s(n)}\ a_n\ c_n^{t'(n)}
\end{split}
\end{equation}
\begin{nota}\label{nota:sum}
For $m \ge n$, we denote $s(m,n) := \sum_{i = n}^ms(i)$,
 $s'(m,n) := \sum_{i = n}^ms'(i)$,
 and $\tau(m,n) := s(m,n) + s'(m,n)$.
We also denote $s(n-1,n) := 0,\ s'(n-1,n) := 0,\ \myand\ \tau(n-1,n) := 0$.
\end{nota}
\begin{nota}\label{nota:small-dmn}
Further, we denote $d_{n,n} := c_n$ and $d_{n+1,n} := c_n^{t(n)}\ a_n\ c_n^{t'(n)}$.
\end{nota}
Then, we can consider
\begin{equation}
 d_{n+1,n}  := \left(d_{n,n}\ e_n^{\tau(n-1,n)}\right)^{t(n)-1}\ d_{n,n}
 \ e_n^{s'(n-1,n)} a_n\ e_n^{s(n-1,n)}
 \ \left(d_{n,n}\ e_n^{\tau(n-1,n)}\right)^{t'(n)-1}\ d_{n,n}
\end{equation}
We denote $a_{n+1,n} := e_n^{s'(n)}\ \fai_{n+1}(a_{n+1})\ e_n^{s(n)}$.
Then, \cref{eqn:descend2-by-d2} can be rewritten as
\[d_{n+2,n} = \left(d_{n+1,n}\ e_n^{\tau(n,n)}\right)^{t(n+1)-1}\ d_{n+1,n}
 \ \ a_{n+1,n}\ \ 
 \left(d_{n+1,n}\ e_n^{\tau(n,n)}\right)^{t'(n+1)-1}\ d_{n+1,n}.
\]
We continue the calculation as follows:
\begin{align*}
 c_{n+3,n} & = \fai_{n+2,n}
 \left(e_{n+2}^{s(n+2)}\ c_{n+2}^{t(n+2)}\ a_{n+2}\ 
  c_{n+2}^{t'(n)}\ e_{n+2}^{s'(n+2)}\right)\\
 &= e_n^{s(n+2)}\ c_{n+2,n}^{t(n+2)}\ \fai_{n+2,n}(a_{n+2})\ c_{n+2,n}^{t'(n+2)} \ e_n^{s'(n+2)}\\
 &= e_{n}^{s(n+2)}\
 \left(e_n^{s(n+1)+s(n)}\ d_{n+2,n}\ e_n^{s'(n+1)+s'(n)}\right)^{t(n+2)}
 \\
 & \qquad \ \fai_{n+2,n}(a_{n+2})
 \left(e_n^{s(n+1)+s(n)}\ d_{n+2,n}\ e_n^{s'(n+1)+s'(n)}\right)^{t'(n+2)}
 \ e_{n}^{s'(n+2)}
\end{align*}
It follows that
\begin{align*}
 c_{n+3,n} & = e_n^{s(n+2,n)}\ 
 \left(d_{n+2,n}\ e_n^{\tau(n+1,n)} \right)^{t(n+2)-1}\ d_{n+2,n}\\
 & \qquad \ e_n^{s'(n+1,n)}\ \fai_{n+2,n}(a_{n+2})\ 
 e_n^{s(n+1,n)}\\
 & \qquad \left(d_{n+2,n}\ e_n^{\tau(n+1,n)} \right)^{t'(n+2)-1}
 \ d_{n+2,n} \ e_n^{s'(n+2,n)}.
\end{align*}
We define $d_{n+3,n}$ by 
\[\fai_{n+3,n} = e_n^{s(n+2,n)}\ d_{n+3,n}\ e_n^{s'(n+2,n)}.\]
Then, $d_{n+3,n}$ is expressed as
\[ \left(d_{n+2,n}\ e_n^{\tau(n+1,n)} \right)^{t(n+2)-1}\ d_{n+2,n}
 \ a_{n+2,n}\ \left(d_{n+2,n}\ e_n^{\tau(n+1,n)} \right)^{t'(n+2)-1}
 \ d_{n+2,n},
\]
where $a_{n+2,n} = e_n^{s'(n+1,n)}\ \fai_{n+2,n}(a_{n+2})\ 
 e_n^{s(n+1,n)}$.
In general, $d_{m,n}$ ($m \ge n+1$) is defined by 
\begin{equation}\label{eqn:basic}
c_{m,n} = e_n^{s(m-1,n)}\ d_{m,n}\ e_n^{s'(m-1,n)}.
\end{equation}
Then,
\begin{equation}\label{eqn:defndmn}
\begin{split}
 d_{m,n} & = \left(d_{m-1,n}\ e_n^{\tau(m-2,n)}\right)^{t(m-1)-1}\ d_{m-1,n}
 \ \ \ a_{m-1,n}\\
 & \qquad \ \left(d_{m-1,n}\ e_n^{\tau(m-2,n)}\right)^{t'(m-1)-1}\ d_{m-1,n},
\end{split}
\end{equation}
 where,
 $a_{m-1,n} = e_n^{s'(m-2,n)}\ \fai_{m-1,n}(a_{m-1})\ e_n^{s(m-2,n)}$.
\begin{rem}\label{rem:a=0}
Suppose that $a_n$ is blank for all $n \bpi$,
 i.e., $c_{n+1,n} = e_n^{s(n)}\ c_n^{t(n) + t'(n)}\ e_n^{s'(n)}$ for all
 $n \bpi$.
In this case, we get
 $d_{m,n} = \left(d_{m-1,n}\ e_n^{\tau(m-2,n)} \right)^{t(m-1)+t'(m-1)-1}\ 
 d_{m-1,n}$.
\end{rem}

Note that, because $s(n,n) < s(n+1,n) < s(n+2,n) < \dotsb$ and
 $s'(n,n) < s'(n+1,n) < s'(n+2,n) < \dotsb$,
 it follows that $\tau(n,n) < \tau(n+1,n) < \tau(n+2,n) < \dotsb$.
For $m > n$,
 a walk $w$ of $G_m$, and $0 \le i < l(w)$, we denote the following:
\itemb
\item $\barw(n,i) = e_n$,
 when $\fai_{m,n}(w(i),w(i+1))$ is $(v_{n,0},v_{n,0})$,
\item $\barw(n,i) = c_n$, otherwise.
\itemn
We denote
 $\barw[n] = (\barw(n,0),\barw(n,1),\barw(n,2),\dotsc,\barw(n,l(w)-1))$.
It is obvious that $\barx(n,i) = c_n$ implies that $\barx(n+1,i) = c_{n+1}$.
\begin{lem}\label{lem:gaps}
Let $m > n > 0$.
In the expression $d_{m,n}$, there exists a walk in the form
 $c_n\ e_n^{\tau(k,n)}\ c_n$ with $n \le k\le m-2$.
There exists a walk in the form $c_n c_n$.
\end{lem}
\begin{proof}
The proof is clear from \cref{eqn:defndmn} and the 
 fact that $t(n) \ge 2$ (or $t'(n) \ge 2$) for all $n \ge 2$.
\end{proof}
\begin{nota}
For $n > 0$ and vertices $u,v \in V(c_n)$,
 there exist unique $0 \le i,j < l(c_n)$ such that $u = v_{n,i}, v = v_{n,j}$.
We denote $\gap(u,v) := j - i$.
For each $x \in X$ and $n > 0$, there exists a unique $0 \le j < l(c_n)$
 such that $x(n,0) = c_n(j)$.
This $j$ is denoted as $j(x,n)$.
Thus, $x(n,0) = c_n(j(x,n)) = v_{n,j(x,n)}$.
Once $j(x,n) \ne 0$, then $j(x,m) \ne 0$ for all $m \ge n$.
Furthermore, because $s(n),s'(n)$ $(n \bpi)$ are positive, it follows that
 $\lim_{m \to \infty}j(x,m) = +\infty$
 and $\lim_{m \to \infty}(l(c_m) -j(x,m)) = +\infty$.
For each $(x,y) \in \left(X\setminus \seb p \sen\right)^2$ with $x \ne y$,
 we denote $\gap_n(x,y) := \gap(x(n,0),y(n,0))$ only for sufficiently large $n$.
Further, $\gap_n(x,y)$ is strictly positive for infinitely many $n$ and/or strictly negative
 for infinitely many $n$.
\end{nota}
The next lemma is evident:
\begin{lem}\label{lem:odd}
Let $n > 0$ be an arbitrary integer.
For arbitrary sequences $s(k), s'(k), t(k), t'(k)$ ($k \ge n$),
 we can adjust the length of $a_k$ ($k \ge n$)
 such that for all $k > n$, $l_k$'s are odd integers.
\end{lem}
\begin{proof}
We omit the proof.
\end{proof}
Let $n \bpi$, $u,v \in V(c_n)$.
Let $m > n$.
We recall that for a walk $w = (u_0,u_1,\dotsc,u_l)$
 and $0 \le a < b \le l$,
 it follows that $w[a,b] = (u_a,u_{a+1},\dotsc,u_b)$.
We denote
\[N_{m,n}(u,v) := \seb l \mid \exists a \ge 0,
\ \ (\fai_{m,n}(c_m))[a,a+l] = (u, \dotsc, v) \sen.\]
\begin{lem}\label{lem:one-and-odd}
Let $m > n \bpi$.
Suppose that for all $m > k \ge n$, $l_k$'s are odd integers,
 and $s(k) = s'(k) = 1$ for all $m > k \ge n$.
We also assume that $2(m-n) > 3l_n$.
Then, for all $u,v \in V(c_n)$, $[3l_n,2(m-n)] \subset N_{m,n}(u,v)$.
\end{lem}
\begin{proof}
By the above lemma, we can assume that for all $m > k \ge n$,
 $l_k$'s are odd.
We recall that for $m > k \ge n$,
 $s(k,n) = s'(k,n) = \sum_{i = n}^k 1 = k- n +1$
 and $\tau(k,n) = s(k,n)+s'(k,n) = 2(k-n+1)$.
We also recall that for each $n' > n$,
\begin{equation}
\begin{split}
 c_{n',n} & = e_n^{s(n'-1,n)}\ \left(d_{n'-1,n}\ e_n^{\tau(n'-2,n)}\right)^{t(n'-1)-1}\ d_{n'-1,n}
 \ \ \ a_{n'-1,n}\\
 & \qquad \ \left(d_{n'-1,n}\ e_n^{\tau(n'-2,n)}\right)^{t'(n'-1)-1}
 \ d_{n'-1,n} \ e_n^{s'(n'-1,n)}.
\end{split}
\end{equation}
In $c_{n'+1,n}$,
 there exists a sub-walk $d_{n',n}\ e_n^{\tau(n'-1,n)}\ d_{n',n}$.
Each $d_{n',n}$ starts with $d_{n'-1,n}$, which starts with $d_{n'-2,n}$
 $\dotsb$.
Consequently, each $d_{n',n}$ starts with $c_n^{t(n)}$.
Similarly, each $d_{n',n}$ ends with $c_n^{t'(n)}$.
Therefore,
 there exists an occurrence $c_n^{t'(n)}\ e_n^{\tau(n'-1,n)}\ c_n^{t(n)}$
 in $d_{n'+1,n}$.
By the assumption that $s(k)=s'(k)=1$ for $m > k \ge n$,
 the integers $2, 4, 6, \dotsc, 2(m-n)$ appear in the form $\tau(k,n)$.
In conclusion, for all integers $\tau = 2, 4, 6, \dotsc, 2(m-n)$,
 $c_{m,n}$ has $c_n^{t'(n)}\ e_n^{\tau}\ c_n^{t(n)}$.
Let $v_{n,i},v_{n,j} \in V(c_n)$.
In the walk $c_n^{t'(n)}\ e_n^{\tau}\ c_n^{t(n)}$,
 there exists a walk
 $w_1 := \underset{*}{\underbracket{c_n}}
 \ e_n^{\tau}\ \underset{**}{\underbracket{c_n}}$.
We denote by $\gap_{w_1}[v_{n,i},v_{n,j}]$ the gap between
 $v_{n,i}$ in $\underset{*}{\underbracket{c_n}}$ and
 $v_{n,j}$ in $\underset{**}{\underbracket{c_n}}$.
Then, $\gap_{w_1}[v_{n,i},v_{n,j}] = l_n - i + \tau + j = l_n +j-i +\tau$.
Let
 $A := \seb l_n +j-i +\tau  \mid \tau = 2, 4, \dotsc, 2(m-n) \sen$.
Because $t'(n) \ge 2$, there exists a walk
 $\underset{*}{\underbracket{c_n}}\ c_n
 \ e_n^{\tau}\ \underset{**}{\underbracket{c_n}}$.
If we take $v_{n,i}$ in $\underset{*}{\underbracket{c_n}}$, then
 we get the set of gaps
 $B :=
 \seb 2l_n +j-i +\tau \mid \tau = 2, 4, \dotsc, 2(m-n) \sen$.
Because $l_n$ is an odd integer,
 $A \cup B \supset [3l_n,2(m-n)]$.
The conclusion is now obvious.
\end{proof}
\begin{prop}\label{prop:topologically-mixing}
Suppose that for all $n > 0$, $l_n \ge 10$ are odd integers
 and $s(n) = s'(n) = 1$ for all $n \bpi$.
Then, $(X,f)$ is topologically mixing.
The number of ergodic measures can be both 1 and 2.
In particular, for Cantor topologically mixing proximal systems of rank 2,
 the number of ergodic measures can be both 1 and 2.
\end{prop}
\begin{proof}
Fix $n \bpi$ as arbitrarily large.
Take $u,v \in V(G_n) = V(c_n)$ arbitrarily.
Then, by \cref{lem:one-and-odd}, $f^k(U(u)) \cap V(v) \nekuu$ for
 all $k \ge 3l_n$.
Because $n, u, v$ are arbitrary, we conclude that $(X,f)$ is topologically
 mixing.
Let us check whether the last statement holds.
Note that, in the previous proposition,
 we have not assumed any condition on $a_n$'s, except 
 that $l_n$'s are odd.
Therefore, by taking $s''(n)$ to be very large compared with
 $t(n)+ t'(n)+ t''(n)$,
 we can make $(X,f)$ uniquely ergodic.
On the other hand, if we make $s''(n)$ small
 compared with  $t(n)+ t'(n)+ t''(n)$,
 we can make $(X,f)$ have two ergodic measures.
\end{proof}
Hereinafter, we consider the cases in which $a_n$ ($n \bpi$) are all blank.
Thus, we have $\fai_{n+1}(c_{n+1}) = e_n^{s(n)}\ c_n^{\bart(n)}\ e_n^{s'(n)}$
 for all $n \bpi$.
We note that for $m > n$,
\[\fai_{m,n}(c_{m}) = e_n^{s(m,n)}\ d_{m+1,n}\ e_n^{s'(m,n)}, \myand\]
\[d_{m+1,n} = \left(d_{m,n}\ e_n^{\tau(m-1,n)}\right)^{\bart(m)-1}\ d_{m,n}. \]
\begin{prop}\label{prop:weakly-mixing-not-topologically-mixing}
There exist sequences $s(n) = s'(n)$ and $\bart(n)$ such that
 $(X,f)$ is not topologically mixing but weakly mixing and uniquely ergodic.
\end{prop}
\begin{proof}
Let $n(1)$ be a positive integer.
Suppose that $G_n$ is constructed for all $n \le n(1)$ such that
 $s(n) = s'(n)$ and $l_1 \ge 3$ is odd.
It is possible to make $l_{n(1)}$ odd by letting $\bart(n)$ be odd
 for all $n < n(1)$.
We remark that $s(n(1)-1) = s'(n(1)-1)$ need not be $1$.
Let $m > n(1)$.
For $k$ with $m > k \ge n(1)$,
 let $s(k) = s'(k) = 1$ and $\bart(k) \ge 3$ be odd.
Then, by \cref{lem:odd}, each $l_k$ $(m \ge k \ge n(1))$ is odd.
Then, for some large $m$, by \cref{lem:one-and-odd},
 we can assume that for all $u,v \in V(c_{n(1)})$,
 $3l_{n(1)} \in  N_{m,n(1)}(u,v)$.
For every $u,v \in V(G_{n(1)})$,
 we get $f^{3l_{n(1)}}(U(u)) \cap U(v)  \nekuu$.
In particular,
 for every $(u_1,v_1),(u_2,v_2) \in V(G_{n(1)}) \times V(G_{n(1)})$,
 we get
 $(f \times f)^{3l_{n(1)}}\left(U(u_1) \times U(v_1)\right)
 \cap U(u_2)\times U(v_2) \nekuu$.
In the above process, we have not defined $s(m),s'(m)$.
We just defined $s(k) = s'(k) = 1$ for $m > k \ge n(1)$.
We take $\bart(m) \ge 3$ to be an odd integer arbitrarily.
We note that
\[\fai_{m+1,n}(c_{m+1}) = e_n^{s(m,n)}\ d_{m+1,n}\ e_n^{s'(m,n)}, \myand\]
\[d_{m+1,n} = \left(d_{m,n}\ e_n^{\tau(m-1,n)}\right)^{\bart(m)-1}\ d_{m,n}. \]
Thus, $l(d_{m+1,n})$ is determined by $l(d_{m,n})$,
 $s(i) = s'(i)$ with $i < m$, and $\bart(m)$.
We take $s(m) = s'(m)$
 to be sufficiently large such that $s(m)+s'(m) > 1.5 \cdot l(d_{m+1,n(1)})$.
We recall that
\[d_{m+2,n} =
 \left(d_{m+1,n}\ e_n^{\tau(m,n)}\right)^{\bart(m+1)-1}\ d_{m+1,n}.\]
Roughly,
 the length of separation of each vertex of $c_n$ is within
 $d_{m+1,n}$ or within
 $d_{m+1,n}\ e_n^{\tau(m,n)}\ d_{m+1,n}$, or it is much greater.
Then, for arbitrary $s(k), s'(k), \bart(k)$ $(k > m)$,
 there exists a constant $K_1 > 1$ such that for every $m' > m$
 and for all $u,v \in V(c_{n(1)})$,
 $[l(d_{m+1,n(1)}) +1,l(d_{m+1,n(1)})+K_1] \cap N_{m',n(1)}(u,v) = \kuu$.
Therefore, $l \in [l(d_{m+1,n(1)})+1,l(d_{m+1,n(1)})+K_1]$
 implies that $f^l(U(u)) \cap U(v) = \kuu$.
Let $n(2) = m+1$, and proceed in the same manner.
We get system $(X,f)$, which is not topologically mixing but weakly mixing.
We show that this construction brings about uniquely ergodic systems.
In the construction, it is obvious that $l_{n+1} > 3l_n$ for all $n \ge n(1)$.
To find the number of ergodic measures, let us compute $1-r(i)$.
We recall that $s(i) = s'(i) = 1$ except when $i = n(k)-1$ for some $k \ge 1$.
Let $n'(k) = n(k)-1$ for all $k \ge 1$.
Suppose that $n(k) \le i < n'(k+1)$ for some $k$.
Then, we get
 $1- r(i) = (s(i) + s'(i))/(s(i)+s'(i) + \bart(i)l_i)
 < 2/(2 + 3l_i) < 2/3l_i < 2/3^i$.
Thus, by \cref{thm:uniquely-ergodic},
 the number of ergodic measures is determined only by the divergence of
\begin{equation}\label{eqn:diverge-weak-mix}
 \sum_{k = 1}^{\infty} (1-r(n'(k))).
\end{equation}
Fix $k$ and let $m = n'(k)$.
We compute
 $1- r(m) = (s(m) + s'(m))/(s(m)+s'(m) + \bart(m)l_{m})$.
If we take $s(m) = s'(m)$ to be large, then
 $1-r(m)$ can be taken arbitrarily close to $1$.
Thus,
 we can make \cref{eqn:diverge-weak-mix} diverge, i.e., $(X,f)$ can be made
 uniquely ergodic.
This concludes the proof.
We would like to show that, by our construction,
 \cref{eqn:diverge-weak-mix} always diverges.
We have assumed that $s(m)+s'(m) > 1.5 \cdot l(d_{m+1,n(k-1)})$.
We let $n = n(k-1)$.
Thus, we get
\begin{align*}
1- r(m) & >  (1.5 \cdot l(d_{m+1,n})/(1.5 \cdot l(d_{m+1,n}) + \bart(m)l_{m})\\
 & =
 (1.5 \cdot l(d_{m+1,n})/(1.5 \cdot l(d_{m+1,n}) + l(d_{m+1,n})+ \tau(m-1,n))\\
 & =  (1.5 \cdot l(d_{m+1,n})/(1.5 \cdot l(d_{m+1,n}) + l(d_{m+1,n})+ 2(m-n))\\
 & >  (1.5 \cdot l(d_{m+1,n})/(1.5 \cdot l(d_{m+1,n}) + 2 \cdot l(d_{m+1,n}))\\
 & =  3/7.
\end{align*}
Thus, \cref{eqn:diverge-weak-mix} diverges.
\end{proof}
\begin{prop}\label{prop:not-weakly-mixing}
There exist sequences $s(n)$, $s'(n)$, and $\bart(n)$ such that
 $(X,f)$ is not weakly mixing.
Both uniquely ergodic systems
 and systems with two ergodic measures are possible.
We can take sequences such that $s(n) = s'(n)$ for all $n \bpi$.
\end{prop}
\begin{proof}
Let $p \ge 3$ be a positive integer.
Let $l_1 \ge 3$ be a multiple of $p$.
Take sequences $s(i),s'(i)$ $(i \bpi)$ such that all $s(i)$ and $s'(i)$ are
 multiples of $p$.
Note that $s(i) = s'(i)$ $(i \bpi)$ is possible.
This guarantees the last statement.
Then, every $l_n$ is a multiple of $p$ for all $n \bpi$.
Fix $n > 1$.
We recall that $c_n = (v_{n,0},v_{n,1},v_{n,2},\dotsc,v_{n,l_n})$.
Then,
 it is easy to see that every two occurrences of $v_{n,1}$ have a gap that is
 a multiple of $p$.
It is the same for $v_{n,2}$.
Therefore,
 $(f \times f)^l(U(v_{n,1}) \times U(v_{n,1})) \cap
 (U(v_{n,1}) \times U(v_{n,2}))
 = \kuu$
 for all $l \ge 0$.
Thus, the first statement is proved.
This construction does not restrict the size of $s(n) + s'(n)$
 and $\bart(n)l_n$.
Thus, it is easy to see that the second statement is valid.
\end{proof}
This brief survey on some properties of proximal Cantor systems with topological rank 2 is no more than just a starting point.
We presented only one concrete example that is mentioned in \cite[Proposition 55]{Blanchard_2008TopSizeOfScrambSet}.
There remain numerous cases of non-primitive substitutions of 2 symbols; some of them must have topological rank 2.
Although proximal Cantor systems with topological rank 2 have some similarities with rank 1 transformations that are considered in the vast field of ergodic theory (cf., for example, \cite{Kal84Kalikow_1984TwofoldMixingImplies}),
 we could not find any link nor identify overlapping systems.

\vspace{5mm}

\noindent
\textsc{Acknowledgments:}
The author would like to thank the anonymous referee(s) for their kind advices.
One of them was pointing out our failure
 not to clarify the coincidence of two topological ranks.
The author would like to thank Editage (www.editage.jp) for providing English-language editing services, before this version was finished.
This work was partially supported by JSPS KAKENHI
(Grant Number 16K05185).

\newcommand{\etalchar}[1]{$^{#1}$}
\providecommand{\bysame}{\leavevmode\hbox to3em{\hrulefill}\thinspace}
\providecommand{\MR}{\relax\ifhmode\unskip\space\fi MR }
\providecommand{\MRhref}[2]{%
  \href{http://www.ams.org/mathscinet-getitem?mr=#1}{#2}
}
\providecommand{\href}[2]{#2}


\end{document}